\documentclass[10pt]{amsart}
\usepackage{graphicx}
\usepackage{amssymb}
\usepackage{amsmath,amsthm}
\usepackage{hyperref}
\usepackage{amsthm}
\usepackage[english]{babel}
\usepackage{mathrsfs}

\newtheorem{thm}{Theorem}[section]
\newtheorem*{thm*}{Theorem}
\newtheorem{cor}[thm]{Corollary}
\newtheorem*{cor*}{Corollary}
\newtheorem{lem}[thm]{Lemma}

\theoremstyle{definition}
\newtheorem{defn}[thm]{Definition}
\newtheorem{rem}[thm]{Remark}

\newcommand{\car}{\curvearrowright}

\newcommand{\G}{\Gamma}
\newcommand{\g}{\gamma}
\newcommand{\e}{\varepsilon}
\def\La{\Lambda}

\def\ra{\rightarrow}
\def\le{\leqslant}
\def\ge{\geqslant}

\newcommand{\N}{{\mathbb N}}

\newcommand{\id}{\operatorname{id}}

\title[Some $OE$ and $W^*$-rigidity results  for actions by wreath product groups]{Some $OE$ and $W^*$-rigidity results
for actions  by wreath product groups}

\author{Ionut Chifan}
\address{Ionut Chifan, Vanderbilt University, 1326 Stevenson Center, Nashville, TN 37240}
\email{ionut.chifan@vanderbilt.edu}
\author{Sorin Popa}
\address{Sorin Popa, UCLA, Math Sciences Building, Los Angeles, CA 90095-1555}
\email{popa@math.ucla.edu}

\author{James Owen Sizemore}
\address{James Owen Sizemore, UCLA, Math Sciences Building, Los Angeles, CA 90095-1555}
\email{sizemore@math.ucla.edu} \subjclass{} \subjclass{}

 \dedicatory{}
\date{\today}
\thanks{The first author's research is partially supported by NSF Grant 1001286}

\thanks{The second author's research is partially supported by NSF Grant DMS-1101718}

\begin{document}

\begin{abstract}
We use deformation-rigidity theory
in von Neumann algebra framework to study probability measure preserving actions by wreath product groups. In particular, we single out large families of wreath products groups satisfying various type of orbit equivalence (OE) rigidity.  For instance, we show that whenever  $H$, $K$, $\G$, $\La$ are icc, property (T) groups such that $H\wr \G$ is measure equivalent to $K\wr \La$ then automatically $\G$ is measure equivalent to $\La$ and $H^{\G}$ is measure equivalent to $K^{\La}$. Rigidity results for von Neumann algebras arising from
certain actions of such groups (i.e. W$^*$-rigidity results) are also obtained.\end{abstract}

\maketitle \tableofcontents


\section*{Introduction and Notations}

The purpose of this paper is to study
rigidity phenomena in von Neumann factors of type II$_1$ and
orbit equivalence relations arising from actions of wreath product groups on
probability measure spaces, by using deformation/rigidity methods.

Rigidity in von Neumann algebras (or W$^*$-{\it rigidity}) occurs
whenever the mere isomorphism of two {\it group measure space}
II$_1$ {\it factor} $L^\infty(X)\rtimes \Gamma\simeq
L^\infty(Y)\rtimes \Lambda$ (or of two {\it group factors}
$L(\Gamma)\simeq L(\Lambda)$), constructed from free, ergodic,
measure preserving actions of countable groups on probability
spaces, $\Gamma \curvearrowright X$,  $\Lambda \curvearrowright Y$
(respectively from infinite conjugacy class groups $\Gamma,
\Lambda$), forces the groups/actions to share some common
properties. The similar type of phenomena in orbit equivalence (OE)
ergodic theory, which derives common properties of the actions
$\Gamma \curvearrowright X$, $\Lambda \curvearrowright Y$ from the
isomorphism of their orbit equivalence relations, is called {\it
OE-rigidity}. These two types of results are in fact closely
related, as any OE of actions implements an

\pagestyle{myheadings} \markboth{I. Chifan, S. Popa, and J.O.
Sizemore}{Some OE and W$^*$-Rigidity Results}

\noindent isomorphism of their associated group measure space von
Neumann algebras (cf. \cite{Si55}), i.e. a W$^*$-equivalence of the
actions. In other words, orbit equivalence is a stronger notion of
equivalence for group actions than W$^*$-equivalence, thus making
W$^*$-rigidity results more challenging to establish than
OE-rigidity. The ultimate purpose for studying such phenomena is, of
course, the classification of group measure space II$_1$ factors and
equivalence relations in terms of their building data $\Gamma
\curvearrowright X$. In this respect, the ``rigidity'' point of view
offers a more suggestive and nuanced terminology, and a far more
intuitive set up.

W$^*$ and OE rigidity can only occur for non-amenable groups, as by
a celebrated result of Connes (\cite{C76}) all II$_1$ factors
$L^\infty(X)\rtimes \Gamma$ with $\Gamma$ amenable are approximately
finite dimensional and thus isomorphic to the so-called {\it
hyperfinite} factor $R$. Similarly, all measure preserving (m.p.)
ergodic actions of amenable groups on the standard probability space
are OE (\cite{OW}, \cite{CFW}). Moreover, non-amenable groups give
rise to non-hyperfinite II$_1$ factors and orbit equivalence
relations. It has been known for some time that non-amenable groups
can produce many classes of non-isomorphic II$_1$ factors and orbit
equivalence relations
(\cite{MvN43,D63,Mc69,Co75,C80,Zi80,Pocorr,CowHaag}), indicating a
very complex picture, and a rich and deep underlying rigidity
theory. But it was during the last ten years that this subject
really took off, with an avalanche of surprising rigidity results
being obtained on both OE and W$^*$ sides.

Much of this is due to the emergence of {\it deformation/rigidity
theory} (\cite{Po01a,PBetti,P1,P-ICM}), a set of techniques that
exploits the tension between ``soft'' and ``rigid'' parts of group
measure space II$_1$ factor $M=L^\infty(X)\rtimes \Gamma$, in order
to recapture the initial data $\Gamma \curvearrowright X$, or part
of it. This approach is based on the discovery that if the group
action has both a ``relatively soft'' part  and a ``relatively
rigid'' part, complementing one another, then the overall rigidity
of the resulting II$_1$  factor $M$ is considerably enhanced. The
``soft spots'' of  an algebra $M$ are gauged by {\it deformations}
by completely positive maps, a prototype of which being {\it
malleable deformations}, that for instance Bernoulli and Gaussian
actions have.

It is due to such a combination/complementarity of ``soft'' and
``rigid'' parts that wreath product groups $G = H\wr \Gamma$ have
soon been recognized to be ``exceptionally rigid'' in the von
Neumann algebra context. Indeed, it was already shown in \cite{P1}
that any isomorphism between group II$_1$ factors $L(G)\simeq
L(G')$, with $G=H \wr \Gamma$, $G'=H'\wr \Gamma'$ wreath product
groups, $H, H'$ abelian and $\Gamma, \Gamma'$ having property (T) of
Kazhdan, forces the groups $\Gamma, \Gamma'$ to be isomorphic. The
same was in fact shown to be true if $\Gamma, \Gamma'$ are
non-amenable product groups (\cite{P-gap}) and for certain
amalgamated free product groups $\Gamma$ (with $\Gamma'$ arbitrary!)
in \cite{PV09}, while in \cite{Io06} it is shown
that for non-amenable ICC groups $H, H'$ and amenable groups $\Gamma,
\Gamma'$, the isomorphism $L(G)\simeq L(G')$ implies $\Gamma \simeq \Gamma'$. Also, II$_1$ factors
$L(G)$ arising from wreath products $G=H \wr \Gamma$ with $H$
amenable and $\Gamma$ non-amenable were shown to be prime in
\cite{P-gap}, a fact that was later strengthened significantly, in
two ways: a relative solidity result for such $L(G)$ is proved in
\cite{CI08}, while a unique prime decomposition result for tensor
products of such factors is obtained in \cite{SW}. Finally, let us
mention that in \cite{IPV}, a large class of generalized wreath
product groups $G$ were shown to be W$^*$-{\it superrigid}, i.e. any
isomorphism between $L(G)$ and the II$_1$ factor $L(G')$ of an
arbitrary group $G'$, forces $G\simeq G'$.

It has been suggested by the second named author in 2007 that a
group measure space factor and orbit equivalence relation arising
from ANY action $G\curvearrowright X$ of a wreath product group $G=H
\wr \Gamma$ may exhibit a certain level of rigidity. This has been
confirmed at the OE-level by Hiroki Sako in \cite{Sa09}, who was
able to prove that for a large class of groups $\Gamma$, the OE
class of an action $H\wr \Gamma\curvearrowright X$ is completely
determined by the OE-class of its restriction $\Gamma
\curvearrowright X$. More precisely, he showed that, if two actions
by wreath products groups are orbit equivalent, $H \wr
\Gamma\cong_{OE}K \wr \La$, where $H$, $K$ are amenable and $\G$,
$\La$ are products of non-amenable, exact groups, then
$\Gamma\cong_{OE}\La$. His methods rely on Ozawa's techniques
involving class $\mathcal S$ groups \cite{Oz,Oz04} thus being
$C^*$-algebraic in nature and depending crucially on exactness of
the groups involved.

In  turn, in this paper we use a deformation/rigidity approach to this problem. This will allow us to
exhibit several large classes of groups for which the OE rigidity phenomenon described above holds. It will also
allow us to obtain some W$^*$-rigidity results of a similar type.

In order to state our OE rigidity result in its full generality, we
recall the following terminology (see e.g. \cite{G3, Fu99}): Two groups
$\Gamma, \La$ are {\it stably orbit equivalent}, or {\it measure
equivalent} ({\it ME}), if there exist free ergodic probability
measure preserving actions $\Gamma \curvearrowright (X, \mu)$,
$\Lambda \curvearrowright (Y, \nu)$, subsets of positive measure
$X_0\subset X$, $Y_0\subset Y$ and an isomorphism of the
corresponding probability spaces $\theta: (X_0, \mu_0) \simeq (Y_0
\nu_0)$ (where $\mu_0=\mu/\mu(X_0)$, $\nu_0=\nu/\nu(Y_0)$), such
that $\theta(\Gamma t \cap X_0)=\Lambda (\theta(t))$, for almost all
$t\in X_0$. We then write $\Gamma\cong_{ME}\La$.

We consider the following three families of groups:  for each $k=1,2,3$, we denote by ${\bf WR}(k)$ the collection of all generalized wreath product groups $H\wr_I\Gamma$ with $\G$ icc, $I$ a $\G$-set with finite stabilizers and satisfying the  condition:

\begin{enumerate}\item \label{ht} $\Gamma$ has property (T) and $H$ has Haagerup's property;
 \item \label{tt}$\Gamma$ and $H$ have property (T) and $H$ is icc;
  \item \label{ap}$\Gamma$ is a non-amenable product of infinite groups and $H$ is amenable.
\end{enumerate}

With this notation,  we obtain the following:

\begin{thm}
Let $H\wr_{I}\G,K\wr_J\La \in{\bf WR}{(k)}$ for some $k=1,2,3$. If $H\wr_{I}\Gamma$ is measure equivalent
to $K\wr_{J}\La$, then  $\Gamma$ follows measure equivalent to $\La$. Moreover, if
$H\wr_{I}\G,K\wr_J\La \in{\bf WR}{(2)}$ is such that $H\wr_{I}\Gamma\cong_{ME}K\wr_{J}\La$, then $\Gamma\cong_{ME}\La$ and
$H^{I}\cong_{ME}K^{J}$.
\end{thm}

To prove the above result, we exploit the fact that the group measure space
von Numann algebra $M$ associated to an action of a wreath product group $H\wr \Gamma$ is ``distinctly soft''
on its $H^{(\Gamma)}$-part, independently of the action. In turn, the fact that $\Gamma$ acts in a very mixing way on $H^{(\Gamma)}$
makes $\Gamma$ ``strongly singular'' (or ``malnormal'') in $M$. When combined with rigidity assumptions on $\Gamma$, this allows us to first extract information about the associated crossed product von Neumann algebra regardless of how the group acts, then finally deducing the above OE rigidity result.

On the other hand, if we now assume that $\G$ acts compactly on the probability space $X$, then we can distinguish the subalgebra
$L(H^{(\Gamma)})$ on which $\G$ acts mixingly from the subalgebra $L^\infty(X)$
on which it acts compactly. This allows us to obtain the following {\it strong W$^*$-rigidity} result:

\begin{thm}
Let $H,K$ be amenable groups and $\G,\La$ groups with the property (T).  Assume that $H\wr\G\curvearrowright^{\sigma}X$ and
$K\wr\La\curvearrowright^{\rho}Y$ are free, measure preserving action such that ${\sigma}_{|\G}$ is compact, ergodic and $\rho_{|\La}$ is
ergodic. If $L^\infty(X)\rtimes (H\wr \G) \simeq L^\infty(Y)\rtimes (K\wr \La)$, then
$\G\curvearrowright^{\sigma_{{|_{\G}}}}X$ is virtually conjugate to $\La\curvearrowright^{\rho_{{|_{\La}}}}Y$.
\end{thm}

\noindent{\bf Organization of the paper.}  In the first section we describe the von Neumann algebras we will be studying and the deformation that we will be using. In the second section we collect various intertwining results concerning subalgebras of von Neumann algebras arising from actions of wreath product groups. The third and fourth section are dedicated to the heart of the deformation/rigidity arguments of the paper, and focus on locating the malnormal, rigid subgroup $\Gamma$ of a wreath product
$H \wr \Gamma$. The final two sections are devoted to the proof of the main theorems.

\vskip 0.1in

\noindent {\bf Notations.} Throughout this paper all finite von Neumann algebras $M$ that we consider are equipped with a normal faithful
tracial state denoted by $\tau$. This trace induces a norm on $N$ by letting $\|x\|_2=\tau(x^*x)^{\frac{1}{2}}$ and $L^2(M)$ denotes the
$\|\cdot \|_2$-completion of $M$. A Hilbert space $\mathcal H$ is a $M$-{\it bimodule} if it carries commuting left and right Hilbert
$M$-module structures.

Given a von Neumann subalgebra $Q\subset M $ we denote by $E_{Q}:M\ra M$ the unique $\tau$-preserving {\it conditional expectation} onto $Q$.
If $e_Q$ is the orthogonal projection of $L^2(M)$ onto $L^2(Q)$ then $\langle M,e_Q\rangle $ denotes the \emph{basic construction}, i.e., the
von Neumann algebra generated by $M$ and $e_Q$ in $\mathcal B(L^2(M))$. The span of $\{xe_Qy \ | \ x,y\in M\}$ forms a dense $*$-subalgebra of
$\langle M,e_Q\rangle$ and there exists a semifinite trace $Tr:\langle N,e_Q\rangle\ra \mathbb C$ given by the formula $Tr(xe_Qy)=\tau(xy)$ for
all $x,y\in M$. We denote by $L^2\langle M,E_Q\rangle$ the Hilbert space obtained with respect to this trace.

The {\it normalizer of $Q$ inside $M$}, denoted $\mathcal N_{M}(Q)$, consists of all unitary elements $u\in \mathcal U(M)$ satisfying
$uQu^*=Q$. A maximal abelian selfadjoint subalgebra $A$ of $M$, abbreviated MASA, is called  a \emph{Cartan subalgebra} if the von Neumann
algebra generated by its normalizer in $M$, $\mathcal N_{M}(A)''$ is equal to $M$. 

If $\G\car^{\sigma}A $ is a trace preserving action by automorphisms of a countable group $\G$ on a finite von Neumann algebra $A$
 we denote by $M=A\rtimes_{\sigma}\G$ the crossed product von Neumann algebra associated with the action. When no confusion will arise we will
 drop the symbol $\sigma$.  Given a subset $F\subset \G$, we will denote by $P_F$ the
 orthogonal projection onto the closure of the span of $\{au_{\g} \ | \ a\in A;\g\in F \}$. 

Given $\omega$ a free ultrafilter on $\mathbb N$ and $(M,\tau)$ a finite von Neumann algebra we denote by $(M^{\omega},\tau^{\omega})$ its
ultrapower algebra, i.e., $M^{\omega}=\ell^{\infty}(\mathbb N ,M)/\mathcal I$ where the trace is defined as
$\tau^{\omega}((x_n)_n)=\lim_{n\ra\omega}\tau(x_n)$ and $\mathcal I$ is the ideal consisting of all $x\in \ell^{\infty}(\mathbb N, M )$ such
that $\tau^{\omega}(x^*x)=0$. Notice that $M$ embeds naturally into $M^{\omega}$ by considering constant sequences. Many times when working
with $M=A\rtimes \G$ we will consider the subalgebra $A^{\omega}\rtimes \G$ of $M^{\omega}$.

For all other notations and terminology, that we may have omitted to
explain in the paper, we refer the reader to \cite{P-gap},
\cite{PV09}, \cite{V-icm}.


\section{Malleable deformations of wreath product groups }\label{sec:wreathproductdeformations}

Let $H$ and $\Gamma$ be two countable discrete groups and assume that $I$ is a $\Gamma$-set. We denote by $H^I=\bigoplus_I H$
the infinite direct sum of $H$ indexed by the elements of $I$, which can also be viewed as the group of finitely supported
$H$-valued function on $I$, with pointwise multiplication. Next consider $\Gamma$ acting on $H^I$ by the generalized Bernoulli
shift i.e. $\rho_g((s_\iota)_{\iota\in I})=(s_{g^{-1}\iota})_{\iota\in I}$ for every $g\in \Gamma$. The corresponding semidirect
product $H^I\rtimes_\rho\Gamma=H\wr_I\Gamma$ is called the \emph{generalized wreath product} of $H$ and $\G$ along $I$. Throughout
this paper, for every $\iota \in I$ we denote its stabilizing group by $\G_{\iota}=\{g\in \G \ | \ g\iota=\iota\}$.

\vskip 0.02in

Given $(A, \tau)$ a finite von Neumann algebra, let $ H\wr_I\Gamma\curvearrowright^{\sigma}(A, \tau)$ be a trace preserving
action and denote by $M=A\rtimes_\sigma (H\wr_I\Gamma)$ the corresponding crossed product von Neumann algebra.
 One important feature of these algebras is that they admit $s$-{\it malleable deformations}, in the general sense of
 \cite{P-ICM}. More specifically, this is obtained as a combination of the Bernoulli-type malleable deformation in
 \cite{Po01a}, \cite{P1} and the free malleable deformations in \cite{Po01a}, \cite{IPP}, being very much in the spirit 
 of the malleable deformation considered in \cite{Io06}. The detailed construction is as follows.

Denote by $\tilde{H}=H*\mathbb{Z}$ and then extend $\sigma$ to an action,
still denoted by $\sigma$, $\tilde{H}\wr_I\Gamma\curvearrowright^{\sigma}(A,\tau)$ by letting the generator $u$ of $\mathbb{Z}$ to act
trivially on $(A,\tau)$. This gives rise to a crossed product von Neumann algebra $\tilde{M}=A\rtimes_{\sigma}(\tilde{H}\wr_I\Gamma)$ and
observe that $M\subset \tilde{M}$.

Seen as an element of $L\mathbb Z$, $u$ is a Haar unitary and therefore one can find a selfadjoint element $h\in L\mathbb Z$ such that
$u=\exp(ih)$. For every $t\in \mathbb {R}$, denote by $u^t=\exp(ith)\in L\mathbb {Z}$ and observe that Ad$(u^t)\in$Aut($L\tilde{H}$). We
further consider the tensor product automorphism $\theta_t=\otimes_I\text{Ad}(u^t)\in \text{Aut}(L \tilde{H}^I)$ and since $\theta_t$ commutes
with $\rho$ then it can be extended to an automorphism of $\tilde{M}$ which acts identically on the subalgebra $A\rtimes_\sigma \Gamma$.

From the definitions one can easily see that $\lim_{t\rightarrow 0}\|u^t-1\|_2=0$ and hence we have $\lim_{t\rightarrow0}\|\theta_t(x)-x\|_2=0$ for
all $x\in \tilde{M}$. Therefore, the path $(\theta_t)_{t\in\mathbb R}$ is a deformation by automorphisms of $\tilde M$.

 Next we show that $\theta_t$
admits a ``symmetry'', i.e. there exists an automorphism $\beta$ of $\tilde{M}$ satisfying the following relations:

\begin{equation}\label{801}\beta^2=1,\text{ } \beta_{|_M}=id_{|_M},\text{ } \beta\theta_t\beta=
\theta_{-t}, \text{ for all } t\in \mathbb {R}. \end{equation}

\noindent To see this, first define $\beta_{|_{LH^I}}=id_{|_{LH^I}}$ and then for every $\iota\in I$ we let $(u)_\iota$ to be the element in
$L\tilde{H}^I$ whose $\iota^{th}$-entry is $u$ and $1$ otherwise. On elements of this form we define $\beta((u)_\iota)=(u^*)_\iota$, and since
$\beta$ commutes with $\rho$, it extends to an automorphism of $L(\tilde{H}\wr_I\Gamma)$ by acting identically on $L\Gamma$. Finally, the
automorphism $\beta$ extends to an automorphism of $\tilde{M}$, still denoted by $\beta$, which acts trivially on $A$. Verifying relations
(\ref{801}) is a straightforward computation and we leave it to the reader.

For further use, we recall that all malleable deformations admitting a symmetry (i.e.
{\it s-malleable} deformations) satisfy the following ``transversality'' property:
\begin{thm}[\cite{P-gap}]\label{transversality} For all $t\in \mathbb R$ and all $x\in M$  we have that
\begin{equation*}\|\theta_{2t}(x)-x\|_2\le 2\|\theta_t(x)-E_M\circ \theta_t(x)\|_2.\end{equation*}

\end{thm}

\section{Intertwining techniques}\label{sec:intertwining}

We review here the techniques of intertwining subalgebras in
\cite{PBetti}, \cite{P1}, which are an essential part of
deformation/rigidity theory. Given a projection $p_0\in M$ and two
subalgebras $P\subset M$ and $Q\subset p_0Mp_0$ one says that
\emph{a corner of $P$ can be embedded into $Q$ inside $M$} if there
exist nonzero projections $p\in P$ $q\in Q$, nonzero partial
isometry $v\in M$ and a $*$-homomorphism $\psi:pPp\rightarrow qQq$
such that $vx=\psi (x)v\text{, }$ for all $x\in pPp$. Throughout
this paper we denote by $P\prec_{M}Q$ whenever this property holds
and by $P\nprec_{M}Q$ otherwise.

\begin{thm}[Popa, \cite{P1}]\label{popaintertwining} Let $(M,\tau)$ be a finite von Neumann algebra with $P\subset M$, $Q\subset M$ two
subalgebras and consider the following properties:
\begin{enumerate}
\item $P\prec_{M}Q$.
\item Given any subgroup $\mathcal{G}\subset \mathcal{U}(P)$ such that $\mathcal{G}''=P$ then for all $x_1,x_2,...,x_n \in M$ and every $\epsilon>0$
there exists $u\in \mathcal{G}$ such that
\begin{equation*}\|E_Q(x_iux_j)\|_2<\epsilon\text{, } \text{ for
every } 1\leqslant i,j\leqslant n.\end{equation*}
\item Given any subgroup $\mathcal{G}\subset \mathcal{U}(P)$ such that $\mathcal{G}''=P$ there exists a sequence $u_n \in \mathcal{G}$
such that \begin{equation*}\lim_{n\rightarrow \infty}\|E_Q(xu_ny)\|_2\rightarrow 0\text{, }\text{ for every } x,y \in M. \end{equation*}

\end{enumerate}

Then one has the following equivalences:  $$non(1) \Leftrightarrow (2)\Leftrightarrow (3)$$
\end{thm}

Based on this criterion, we present below a few intertwining lemmas needed in the coming sections. The first result we prove deals with
embedding of normalizers and will be used quite extensively in Section 5. Roughly speaking, given $Q$ a regular subalgebra of $ M$ with
$Q\subseteq N\subseteq M$ and $\mathcal{G}$ a subgroup of normalizers of $Q$ in $M$, if there exits a nonzero partial isometry intertwining
$\mathcal{G}''$ into $N$ then one can find a nonzero partial isometry in $M$ intertwining the (possibly larger) algebra $(\mathcal U(Q)\mathcal
G)''$ into $N$. The precise statement is the following:

\begin{lem}\label{augumentintertwining} Let $Q\subseteq N \subseteq M $ be finite von Neumann algebras such that $\mathcal{N}_{M}(Q)''=M$.
If $\mathcal{G}\subset\mathcal{N}_{M}(Q)$ is a subgroup such that $\mathcal{G}''\prec_M N$ then $(\mathcal{U}(Q)\mathcal{G})''\prec_M N$.
\end{lem}
\begin{proof} Suppose by contradiction that we have $(\mathcal{U}(Q)\mathcal{G})''\nprec_M N$. Therefore, by Theorem \ref{popaintertwining},
there exists an infinite sequence $x_n=a_nu_n\in \mathcal{U}(Q)\mathcal{G}$ with $ a_n\in \mathcal{U}(Q)$ and $u_n\in\mathcal{G}$ such that
\begin{equation}\label{61} \lim_{n\rightarrow \infty}\|E_N(xx_ny)\|_2= 0 \text{ for all } x,y \in M.
\end{equation}
Taking $x=y=1$ in (\ref{61}) it is immediate that the sequence $(u_n)_n$ must be infinite. Below we prove that
\begin{equation}\label{60} \lim_{n\rightarrow \infty}\|E_N(xu_ny)\|_2= 0 \text{ for all } x,y \in M.
\end{equation}

Fix two arbitrary unitaries $x,y\in \mathcal{N}_{M}(Q)$. Then for all $a_n$ we have $xa_nx^* \in \mathcal{U}(Q)\subset N$ and using (\ref{61})
we deduce that:
\begin{equation*} \lim_{n\rightarrow \infty}\|E_N(xu_ny)\|_2=\lim_{n\rightarrow \infty}\|xa_nx^*E_N(xu_ny)\|_2=\end{equation*}
\begin{equation*}=\lim_{n\rightarrow \infty}\|E_N(xa_nx^*xu_ny)\|_2=\lim_{n\rightarrow \infty}\|E_N(xx_ny)\|_2= 0. \end{equation*}
The above convergence extends to all elements $x, y$ that are finite linear combinations of unitaries in $\mathcal{N}_{M}(Q)$ and furthermore,
using $\|\cdot\|_2$-approximations, to all elements $x,y$ belonging to $\mathcal{N}_M(Q)''$. Since $\mathcal{N}_M(Q)''=M$, this completes the
proof of (\ref{60}).

Finally, by Theorem \ref{popaintertwining} convergence (\ref{60}) implies  that $\mathcal{G}''\nprec_M N$ thus leading to a contradiction.
\end{proof}

The next lemma is more specialized, providing a criterion for
intertwining certain subalgebras inside von Neumann algebras arising
from actions by wreath product groups. In essence the result is a
translation of Theorem \ref{popaintertwining} in the setting of
ultrapower algebras and we include a proof only for the sake of
completeness. The reader may also consult Section 3 in \cite{P2} or
Proposition 2.1 in \cite{CP10}, for a similar arguments.

\begin{lem}\label{intertwining}Let $H\wr_I\G\car A$ be a trace preserving action on a finite von Neumann algebra $A$.
Denote by $M=A\rtimes (H\wr_I \G) $ and let $P\subset M$ be a II$_1$ subfactor such that $\mathcal N_{M}(P)'\cap M=\mathbb C 1$. If $S\subset
I$ is a subset, then  $P\prec_{M}A\rtimes H^S$ implies $P^{\omega}\subseteq (A\rtimes H^S)^{\omega}\vee M.$ When assuming $S=I$ the two
conditions are actually equivalent.

\end{lem}

\begin{proof}

Assume $P\prec_{M}A\rtimes H^S$. Therefore one can find nonzero projections $p\in P$, $q\in A\rtimes H^S$, a $*$-homomorphism $\psi: pPp\ra
q(A\rtimes H^S)q$ and nonzero partial isometry $v\in M$ such that $v\psi(x)=xv$ for all $x\in pPp$. The last equation implies that $vv^*\in (
pPp)'\cap pMp $ and therefore we have the following
\begin{equation}\label{802}pPpvv^*=v\psi(pPp)v^*\subseteq v(A\rtimes H^S)v^*.\end{equation}

We notice that there exists nonzero projection $p'\in P'\cap M$ such that $vv^*= pp'$ and combining this with (\ref{802}) we obtain
\begin{equation}\label{101}(pPp)^{\omega}p'\subseteq (A\rtimes H^S)^{\omega}\vee M
.\end{equation}

Since $P$ is a II$_1$ factor then after shrinking the projection $p$ if necessary one may assume that $p$ has trace $\frac{1}{k}$, for some
positive integer $k$. Also, for every $1\le i,j\le k$ there exist partial isometries $e_{ij}\in P $ such that $e_{11}=p$, $e^*_{ij}=e_{ji}$,
$e_{ij}e_{ji}=e_{ii}\in \mathcal P(P)$ and $\sum_i e_{ii}=1$. If $(x_n)_n\in P^{\omega}$ then using the above relations in combination with
$p'\in P'\cap M$ we
have \begin{eqnarray}\nonumber(x_n)_n(p')_n&=&(x_np')_n=(\sum_{i,j}e_{ii}x_ne_{jj}p')_n=\sum_{i,j}(e_{i1}e_{1i}x_ne_{j1}e_{1j}p')_n\\
\label{102}&=&\sum_{i,j}(e_{i1})_n(e_{1i}x_ne_{j1})_n(p')_n(e_{1j})_n.\end{eqnarray}

One can easily see that $(e_{1i}x_ne_{j1})_n\in (pPp)^{\omega}$ and combining this with (\ref{101}) and (\ref{102}) we conclude that
$(x_n)_n(p')_n\in (A\rtimes H^S)^{\omega}\vee M$, thus showing that $P^{\omega}p'\subseteq (A\rtimes H^S)^{\omega}\vee M$.

Conjugating by $u\in\mathcal N_M(P)\subseteq \mathcal N_M(P'\cap M)$ we obtain $P^{\omega}up'u^*\subseteq (A\rtimes H^S)^{\omega}\vee M$, for
all $u\in\mathcal N_M(P)$, and hence $P^{\omega}p_0\subseteq (A\rtimes H^S)^{\omega}\vee M$ where $p_0=\vee_{u\in\mathcal N_M(P)}up'u^*\in
P'\cap M$. It is clear that $p_0$ commutes with $\mathcal N_M(P)$ and thus it belongs to $\mathcal N_M(P)'\cap(P'\cap M)$. By assumption we
have $\mathcal N_M(P)'\cap M=\mathbb C 1$ which forces $p_0=1$ and therefore $P^{\omega}\subseteq (A\rtimes H^S)^{\omega}\vee M$.

For the converse we proceed by contraposition, i.e., assuming $S=I$ we show that $P\nprec_{M}A\rtimes H^I$ implies $P^{\omega}\nsubseteq
(A\rtimes H^I)^{\omega}\vee M$. If $P\nprec_{M}A\rtimes H^I$, by Theorem \ref{popaintertwining}, there exists a sequence of unitaries
$a_n\in\mathcal U (P)$ such that for all $x,y\in M$ we have $\|E_{A\rtimes H^I}(xa_ny)\|_2\ra 0 $ as $n \ra \infty$. This implies $a\perp
M(A\rtimes H^I)^{\omega}M$, where $a=(a_n)_n\in P^{\omega}$ and since $M(A\rtimes H^I)^{\omega}M=(A\rtimes H^I)^{\omega}\vee M$ we conclude
that $P^{\omega}\nsubseteq (A\rtimes H^I)^{\omega}\vee M$.
\end{proof}

 In the following lemma we collect three situations when we have good control over intertwiners between certain subalgebras of
  von Neumann algebras arising from actions of wreath product
 groups. The result is a mild extension of Theorem 3.1 in \cite{P1}, and has exactly the same proof, which however we
 include here for the reader's convenience.

\begin{lem}\label{controlnorm}
Let $H\wr_{I}\G \curvearrowright^{\sigma}(A,\tau)$ ba a trace preserving action on a finite algebra A. Denote by $\tilde{M}=A\rtimes
_{\tilde{\sigma}}(\tilde{H}\wr_{I}\G)$, $M=A\rtimes_\sigma (H\wr_{I}\G)$ and $P=A\rtimes\G$.
\begin{enumerate}
\item Let $q\in M$ be a projection and
$Q\subset qMq$ be a von Neumann subalgebra. Assume that for every $\iota\in I$ one has $Q\nprec_{M}A\rtimes (H\wr_I \G_{\iota})$. If
$0\neq\xi\in L^2(q\tilde{M})$ satisfies $Q\xi\subset L^2(\sum_i \xi_iM)$ for some $\xi_1, ..., \xi_n\in L^2(\tilde{M})$ then $\xi\in L^2(M)$;
in particular we have $Q'\cap q\tilde{M}q\subseteq\mathcal{N}_{q\tilde{M}q}(Q)'' \subseteq M$.

\noindent If $I$ has finite stabilizers and $Q\subset qMq$ such that $Q\nprec_{M}A\rtimes H^{I}$ then we have $Q'\cap
q\tilde{M}q\subseteq\mathcal{N}_{q\tilde{M}q}(Q)\subseteq qMq$.

\item Let $q\in P$ be a projection and
$Q\subset qPq$ be a von Neumann subalgebra. Assume that for every $\iota\in I$ one has $Q\nprec_{P}A\rtimes
 \G_{\iota}$. If $0\neq\xi\in
L^2(qM)$ satisfies $Q\xi\subset L^2(\sum_i \xi_iP)$ for some $\xi_1, ..., \xi_n\in L^2(M)$ then $\xi\in L^2(P)$; in particular we have $Q'\cap
qMq\subseteq \mathcal{N}_{qMq}(Q) \subseteq P$.

\noindent If $I$ has finite stabilizers and $Q\subset qPq$ such that $Q\nprec_{P}A$ then we have $Q'\cap
qMq\subseteq\mathcal{N}_{qMq}(Q)''\subseteq qPq$.

\item Assume that $I$ has finite stabilizers and let $S \subset I$ a finite subset. If $Q\subset A\rtimes (H\wr_{\G_\iota S} \G_{\iota})$ is a subalgebra such
that $Q\nprec_MA$ then we have

\begin{equation*}\mathcal{N}_M(Q)''\prec_M A\rtimes H^I, \end{equation*}
\end{enumerate}
\end{lem}

\begin{proof}

 Let $p$ denote the orthogonal projection of $L^2(M)$ onto the Hilbert subspace $\overline{Q\xi M}^{\|\cdot\|_2}\subset
L^2(\tilde{M})$. Note that $p\in Q'\cap q\langle \tilde{M}, e_{M}\rangle q$ and $0\neq Tr(p) < \infty$, where $Tr$ denotes the canonical trace
on $\langle \tilde{M}, e_{M}\rangle.$ To prove that $\xi\in M$ it is sufficient to show that $p\leq e_{M}$ or, equivalently, $(1 - e_{M})p(1 -
e_{M}) = 0$.

By taking spectral projections, to show that $(1 - e_{M})p(1 -
e_{M}) = 0$ it is in fact sufficient to show that if $f\in
Q'\cap\langle M, e_{P}\rangle$ is a projection such that $0\neq
Tr(f) < \infty$ and $f\le1-e_{M}$ , then $f = 0$. To this end, we
will show that $\|f\|_{2,Tr}$ is arbitrarily small.

Thus, let $\tilde{\eta}_0 = e$ and let $\tilde{\eta}_1,..., \tilde{\eta}_n,...$ be an enumeration of elements in $(\tilde{H}\setminus H)^I$
which are representatives for left cosets of $H\wr_{I}\G$ in $\tilde{H}\wr_I\G$. Next if we denote by $f_n =
\sum^n_{i=1}u_{\tilde{\eta}_i}e_{M}u_{\tilde{\eta}_i^{-1}}$ then, as $f$ has finite trace and $f\le 1 - e_M =
\sum^\infty_{i=1}u_{\tilde{\eta}_i}e_{M}u_{\tilde{\eta}_i^{-1}}$, there exists $n\in\N$ such that $\|f_nf - f\|_{2,Tr} < \epsilon\|f\|_{2,Tr}.$
Thus, if $u\in\mathcal{U}(Q)$ then

\begin{equation}\label{1} Tr(f_nuf_nu^*)\ge
Tr(ff_nfuf_nu^*)-|Tr(ff_n(1-f)uf_nu^*)|-|Tr((1-f)f_nuf_nu^*)|.\end{equation}

Since $f_nf$ is $\epsilon-$close to $f$ in the norm $\|\cdot\|_{2,Tr}$ and $f$ commutes with $u\in Q$ we deduce:

\begin{equation}\label{2} Tr(ff_nfuf_nu^*) =
Tr(f_nfuf_nfu^*)\ge(1-2\epsilon-\epsilon^2)\|f\|^2_{2,Tr}.\end{equation}

Similarly, we have:

\begin{equation*}|Tr(ff_n(1-f)uf_nu^*)| + |Tr((1-f)f_nfuf_nu^*)|\le 2\epsilon(1 +
\epsilon)\|f\|^2_{2,Tr}.\end{equation*}

Combining this with (\ref{1}) and (\ref{2}) we get that for all $u\in\mathcal{U}(Q)$ we have

\begin{equation}\label{3} Tr(f_nuf_nu^*) \ge(1 - 4\epsilon - 3\epsilon^2)\|f\|^2_{2,Tr}.\end{equation}

On the other hand a straight forward computation shows that

\begin{equation}\label{4} Tr(f_nuf_nu^*) = Tr(\sum_{i,j}u_{\tilde{\eta}_i}e_{M}u_{\tilde{\eta}_i^{-1}}uu_{\tilde{\eta}_j}e_{M}u_{\tilde{\eta}_j^{-1}}u^*) = \sum_{i,j} \|E_{M}(u_{\tilde{\eta}_i}uu_{\tilde{\eta}_j^{-1}})\|_2^2.\end{equation}

Thus, in order to prove that $\|f\|_{2,Tr}$ is small, it suffices to show that for every $\tilde{\eta}_1,\ldots , \tilde{\eta}_n\in
(\tilde{H}\setminus H)^I$ and every $\epsilon> 0$ there exists $u\in\mathcal{U}(Q)$ such that for all $1\le i,j\le n$ we have

\begin{equation}\label{35}\|E_{M} (u_{\tilde{\eta}_i}uu_{\tilde{\eta}^{-1}_j})\|_2\le\epsilon.\end{equation}

Fix an $\varepsilon>0$ and an arbitrary set $\{\tilde{\eta}_1,\tilde{\eta}_2,\tilde{\eta}_3,\ldots ,\tilde{\eta}_n\}$. For every $1\le i\le n$
denote by $\tilde{F}_i$ the support of $\eta_i$ and let $\mathcal{F}=\bigcup^n_{i=1}\tilde{F}_i\subset I$. It is easy to see that for every
$1\le i,j\le n$ we have the following containment
\begin{equation}\label{31} \{g\in \G \ | \ g\tilde{F}_j=\tilde{F}_i\}\subseteq\bigcup_{\kappa,\ell\in\mathcal{F}}\{g\in \G \ | \ g\kappa=\ell\}.
\end{equation}
Furthermore, observe that $\{g\in \G \ | \ g\kappa=\ell\}$ is either empty or equal to $g_{\kappa,\ell}\G_{\kappa}$ for a fixed element
$g_{\kappa,\ell}\in \G$ satisfying $g_{\kappa,\ell}\kappa=\ell$. When combined with (\ref{31}) it implies that for every $1\le i,j\le n$ we
have
\begin{equation*}
\{g \in \G \ | \ g\tilde{F}_j=\tilde{F}_i\}\subseteq \bigcup_{\kappa,\ell\in\mathcal{F}}g_{\kappa,\ell}\G_{\kappa},
\end{equation*}
which further implies that \begin{equation}\label{39}\{\eta g \in H\wr_I \G \ | \ \eta \in H^I; g\tilde{F}_j=\tilde{F}_i\}\subseteq
\bigcup_{\kappa,\ell\in\mathcal{F}}g_{\kappa,\ell}(H\wr_I \G_{\kappa}).
\end{equation}

Since $\mathcal{F}$ is a finite set and for every $\kappa\in \mathcal{F}$ we assumed that $Q\nprec_{P} A\rtimes_{\sigma}(H\wr_I \G_{\kappa})$
then by Theorem \ref{popaintertwining} there exists a unitary $u_{\mathcal{F},\varepsilon}\in \mathcal{U}(Q)$ such that, for all $\kappa,\ell
\in \mathcal{F}$ we have
\begin{equation*}\|E_{A\rtimes_{\sigma}H^I\rtimes \G_{\kappa}}(u_{g_{\kappa,\ell}^{-1}}u_{\mathcal{F},\varepsilon})\|_2\le\frac{\varepsilon}{|\mathcal{F}|}.\end{equation*}

Using the Fourier expansion $u_{\mathcal{F},\varepsilon}=\sum_{\eta g\in H\wr_{I}\G} a_{\eta g}u_{\eta g}$, a little computation shows that the
above inequality is equivalent to
\begin{equation}\label{37}
\sum_{\eta g\in g_{\kappa,\ell}(H\wr_I \G_{\kappa})}\|a_{\eta g}\|_2^2\le\frac{\varepsilon^2}{|\mathcal{F}|^2}\text{ for all }\kappa,\ell\in
\mathcal{F}.
\end{equation}

\noindent Next we show that the unitary $u_{\mathcal{F},\varepsilon}\in Q$ found above satisfies (\ref{35}). Indeed, employing the formula for
the conditional expectation, we obtain

\begin{equation*}\|E_{M}(u_{\tilde{\eta}_i}u_{\mathcal{F},\varepsilon}u_{\tilde{\eta}^{-1}_j})\|^2_2=\sum_{\{\eta g|\tilde{\eta}_i\eta \rho_g(\tilde{\eta}_j^{-1})\in
H^{I}\}}\|\tilde{\sigma}_{\tilde{\eta}_i}(a_{\eta g})\|^2_2\le\sum_{\{\eta g|\eta \in H^I; g\tilde{F}_j=\tilde{F}_i\}}\|a_{\eta
g}\|^2_2,\end{equation*}

\noindent and combining this with (\ref{39} ) and (\ref{37}) we have
\begin{equation*}
\|E_{P}(u_{\eta_i}uu_{\eta^{-1}_j})\|^2_2\le\sum_{\kappa,\ell\in\mathcal{F}}\sum_{\eta g\in g_{\kappa,\ell}(H\wr_I \G_{\kappa})}\|a_{\eta
g}\|^2 \le\sum_{\kappa,\ell\in \mathcal{F}}\frac{\varepsilon^2}{|\mathcal{F}|^2}=\varepsilon^2,
\end{equation*}
which finishes the proof of (1). \vskip 0.05in

The proof of the part (2) is very similar with the first one and it will be omitted.

\vskip 0.05in

Below we prove part (3). Let $\tilde {K}=\{g\in \G|\text{ exist } x,y\in \G_{\iota}S \text{ such that } g x=y\}$.

First observe that since $Q \nprec_M A$ and $\G_\iota$ is finite then $Q\nprec_MA\rtimes \G_{\iota}$. Therefore, by Theorem
\ref{popaintertwining}, there exists a sequence of unitaries $x_n\in Q$ such that for all $z,t\in M$ we have

\begin{equation*}\label{82} \lim_{n\rightarrow \infty}\|E_{A\rtimes \G_{\iota}}(zx_nt)\|_2\rightarrow 0
\end{equation*}

\noindent Using Fourier expansion we have $x_n=\sum_{\eta h}b^n_{\eta h}u_{\eta h}\in Q$ and therefore the above convergence is equivalent to
the following \begin{eqnarray}\label{84}\sum_{h \in \G_{\iota} }\|b^n_{\delta h \lambda}\|^2_2 \rightarrow 0\text{ for every }\delta, \lambda
\in \G.\end{eqnarray}

Next we prove that for all $c,d \in A\rtimes_{\sigma}(H^I)$ and $g\in\G\setminus\tilde{K}$ we have

\begin{equation*}
\lim_{n\rightarrow \infty}\|E_{A\rtimes_{\sigma}(H\wr_{\G_{\iota}F}\G_{\iota})}(cu_{g}x_nu_{\gamma^{-1}}d)\|_2 \rightarrow 0.
\end{equation*}


Using $\|\cdot\|_2$- approximations it suffices to show our claim only for elements of the form $c=c_1u_{\eta_1}$, $d=c_2u_{\eta_2 }$, where
$c_{1,2}\in A$ and $\eta_{1,2}\in H^I$. Therefore, using the expansion $x_n=\sum_{\eta h}b^n_{\eta h}u_{\eta h}\in N$, we have that:

\begin{equation*}\|E_{A\rtimes_{\sigma}(H\wr_{\G_{\iota}S}\G_{\iota})}(cu_{g}x_nu_{\gamma^{-1}}d)\|_2=
\end{equation*}

\begin{eqnarray*}
 &=&\|\sum_{\eta h}E_{A\rtimes_{\sigma}(H\wr_{\G_{\iota}F}\G_{\iota})}(c_1u_{\eta_1}u_{g}b^n_{\eta h}u_{\eta
h}u_{\gamma{-1}}c_2u_{\eta_2})\|_2\\&=& \|\sum_{\stackrel{\eta h\in H\wr_{\G_{\iota}F}\G_{\iota}}{\eta_1g\eta h \gamma^{-1}\eta_2\in
H\wr_{\G_{\iota}F}\G_{\iota}}} c_1u_{\eta_1}u_{g}b^n_{\eta h}u_{\eta h}u_{\gamma^{-1}}c_2u_{\eta_2}\|_2
\end{eqnarray*}

Since $\eta h\in H\wr_{\G_{\iota}S}\G_{\iota}$, we observe that condition $\eta_1g\eta h \gamma^{-1}\eta_2\in H\wr_{\G_{\iota}S}\G_{\iota}$ is
equivalent to $gh\g^{-1}\in \G_\iota$ and $\eta_1\rho_{g}(\eta)\rho_{g h \gamma^{-1}}(\eta_2)\in H^{\G_{\iota}S}$. Since $g\in \G \setminus
\tilde{K}$, then the later condition is equivalent to the following: There exist at most finitely many $\eta^1_k$, subwords of $\eta _1$ and
finitely many $\eta^2_l$ subwords of $\eta_2$, such that $\eta_1\rho_{g}(\eta)\rho_{g h \gamma^{-1}}(\eta_2)=e$. This is furthermore equivalent
with $\eta=\rho_{g{-1}}((\eta^{1}_k)^{-1})\rho_{ h \gamma^{-1}}((\eta^2_l)^{-1})$ and hence the above sum is equal to:
\begin{eqnarray*}
&&\|\sum_{\stackrel{h\in \G_{\iota}}{\eta=\rho_{g^{-1}}((\eta^{1}_k)^{-1})\rho_{ h
\gamma^{-1}}((\eta^2_l)^{-1})}; k,l} c_1u_{\eta_1}u_{g}b^n_{\eta h}u_{\eta h}u_{\gamma^{-1}}c_2u_{\eta_2}\|^2_2 \\
&\leq &\|c\|^2\|d\|^2\|\sum_{\stackrel{h\in \G_{\iota}}{\eta=\rho_{g^{-1}}((\eta^{1}_k)^{-1})\rho_{ h \gamma^{-1}}((\eta^2_l)^{-1});k,l}}
b^n_{\eta h}u_{\eta h}\|^2_2
\end{eqnarray*}
\begin{equation}\label{83}
 = \|c\|^2\|d\|^2\sum_{h\in \G_{\iota};k,l} \|b^n_{\rho_{g{-1}}((\eta^{1}_k)^{-1}) h
\gamma^{-1}(\eta^2_l)^{-1}\gamma}\|^2_2
\end{equation}

Since $\eta^1_k$ and $\eta^2_l$ are finite sets depending only on $c,d ,g, \gamma$ (which were fixed!) then by (\ref{84}) the sum (\ref{83})
converges to $0$ when $n\rightarrow \infty$ thus finishing the proof of the claim.

Now we continue with the proof. We proceed by contradiction so assume that $\mathcal N_M(Q)''\nprec_M A\rtimes H^I$. Fix $\e>0$ and by Theorem
\ref{popaintertwining} there exists a unitary $u=\sum_{g\in\G} a_gu_g\in\mathcal{N}_M(Q)$, with $a_g\in A\rtimes H^I$, such that
$\sum_{g\in\tilde{K}}\|a_g\|_2<\epsilon$. Furthermore, we can find a finite set $K\subset\G$ such that $\sum_{g\in\G\setminus
K}\|a_g\|^2_2<\epsilon$.

Denoting by $v=\sum_{g\in K\setminus\tilde{K}}a_gu_g$ the above inequalities imply that $\|ux_nu^*-vx_nu^*\|_2<2\epsilon$. Using this in
combination with $x_n\in A\rtimes(H\wr_{\G_{\iota}F}\G_{\iota})$ and $u\in \mathcal N_M(Q)$, a straight forward computation  shows that
\begin{equation}\label{803}\|E_{A\rtimes(H\wr_{\G_{\iota}F}\G_{\iota})}(vx_nu^*)\|_2>1-2\epsilon.\end{equation}

On the other hand we have

\begin{equation*}\|E_{A\rtimes(H\wr_{\G_{\iota}F}\G_{\iota})}(vx_nu^*)\|_2=\|\sum E_{A\rtimes(H\wr_{\G_{\iota}F}\G_{\iota})}(a_gu_gb_{\eta h}u_{\eta
h}u_\gamma^*a_\gamma^*)\|_2.\end{equation*}

Notice that, if a term in the sum above is nonzero we must have that $gh\gamma^{-1}\in\G_{\iota}$, where $g\in K\setminus\tilde{K},
h\in\G_{\iota}$. Since $K$ and $\G_{\iota}$ is finite, this means that only finitely many $\gamma$ will contribute to the sum. By our claim
above for each $g\in K\setminus\tilde{K}$ and $\gamma\in\G$ we know the above norm converges to $0$. Since there are only finitely many such
$g$ and $\gamma$ we get that

$$\|E_{A\rtimes(H\wr_{\G_{\iota}F}\G_{\iota})}(vx_nu^*)\|_2\rightarrow0.$$

\noindent This however, contradicts (\ref{803}) when letting $\epsilon$ to be sufficiently small.

\end{proof}


\section{Rigid Subalgebras of M}\label{sec:rigidsubalg}

In this section we come to the heart of the deformation/rigidity arguments of the paper. The central idea, as usual in Popa's
deformation/rigidity theory, is to use deformations to reveal the position of rigid subalgebras of von Neumann algebras $M$ arising from
actions by wreath product groups. More precisely, our main result shows that if the deformation $\theta_t$ introduced in the first section
converges uniformly to the identity on the unit ball of a diffuse subalgebra $Q$ then one can completely determine the position $Q$ inside $M$.
One consequence we derive from this is Theorem \ref{intertwiningrigid} describing all rigid diffuse subalgebras of $M$.

This result is very much in the spirit of Theorem 4.1 of \cite{P1}
and Theorem 3.6 of \cite{Io06} and in fact most of our proofs
resemble the proofs of these results. Roughly speaking, the methods
we use, employ averaging arguments in combination with the
intertwining techniques described in the previous section. \vskip
0.05in
 The following technical result can be seen as a criterion for locating subalgebras inside von Neumann algebras $M$ arising from
actions by wreath product groups.
\begin{thm}\label{intertwiningrigid2} Let $H,\G$ be countable groups and let $I$ a $\G$-set with finite stabilizers.
Let $H\wr_{I}\Gamma\curvearrowright A$ be a trace preserving action on a finite algebra $A$ and denote by $M=A\rtimes( H\wr_{I}\Gamma)$. If
$Q\subset pMp$ is a diffuse subalgebra such that $\theta_t\ra id$ uniformly on the unit ball of $Q$, then one of the following alternatives
holds:
\begin{enumerate}
\item  $Q\prec_M A\rtimes\Gamma$,
\item There exists $\iota \in I$ and a finite set $F\subset I$ such that $Q\prec_M
A\rtimes_{\sigma} (H\wr_{\G_{\iota}F} \G_{\iota})$.
\end{enumerate}

\end{thm}

The proof of this theorem will result from a  sequence of lemmas.
The first one is  taken from \cite{Po01a}, \cite{P1}, but we
include a proof for completeness.

\begin{lem}\label{rigidaltern} Let $H,\G$ be countable groups and let $I$ a $\G$-set with finite stabilizers.
Let $H\wr_{I}\Gamma\curvearrowright A$ be a trace preserving action on a finite algebra $A$ and denote by $M=A\rtimes( H\wr_{I}\Gamma)$. If
$Q\subset pMp$ is a diffuse subalgebra such that $\theta_t\ra id$ uniformly on the unit ball of $Q$, then one of the following alternatives
holds:
\begin{enumerate}
\item  There exists a nonzero partial isometry $w\in \tilde M$ such that $\theta_1(x)w=wx$
for all $x\in Q$.

\item  There exists $\iota\in I$ such that $Q\prec_M A\rtimes (H\wr_I \G_{\iota})$.
\end{enumerate}
\end{lem}

\begin{proof} 
Since $\theta_t\ra id $ uniformly on the unit ball of $Q$ we can find $n\ge1$ such that \begin{equation*}\|\theta_{1/2^n}(u) - u\|\le1/2,\text{
for all }u\in\mathcal{U}(Q).\end{equation*}

Let $v$ be the minimal $\|.\|_2$ element of $\mathcal{K} = \overline{co}^w\{\theta_{1/2^n}(u)u^*|u\in\mathcal{U}(Q)\}$. Since
$\|\theta_{1/2^n}(u)u^* - 1\|_2 = 1/2$, for all $u\in\mathcal{U}(Q)$, we get that $\|v - 1\|_2 = 1/2$, thus $v\neq 0$. Also, since
$\theta_{1/2^n}(u)\mathcal{K}u^* = \mathcal{K}$ and $\|\theta_{1/2^n}(u)xu^*\|_2 = \|x\|_2$, for all $u\in\mathcal{U}(Q)$, the uniqueness of
$v$ implies that $\theta_{1/2^n}(u)v = vu$ for all $u\in\mathcal{U}(Q)$ and hence
\begin{equation}\label{5} \theta_{1/2^n}(x)v = vx, \text{ for all } x\in Q.
\end{equation}
Assume that (2) is false, i.e. for every $\iota\in I$ one has
$Q\nprec_{M}A\rtimes(H\wr_I \G_{\iota})$. Therefore part (1) of
Lemma \ref{controlnorm} implies that $Q'\cap \tilde{M}\subset M$. On
the other hand, since $\theta_t$ is a $s$-malleable deformation then
combining (\ref{5}) with the procedure from \cite{P1} of patching up
intertwiners, one can find a non-zero partial isometry
$w\in\tilde{M}$ such that $\theta_1(u)w = wu$, for all
$u\in\mathcal{U}(Q)$, which proves (1).
\end{proof}

Our  second lemma is a refinement of arguments in Section 4 of
\cite{P1}.

\begin{lem}\label{positionreading}

Let $M$ and $\tilde{M}$ be as above. Assume $Q\subset pMp $ is a von Neumann subalgebra such that there exists a nonzero partial isometry
$v\in\tilde{M}$ satisfying that $\theta_1(x)v = vx$, for all $x\in Q$. Then one of the following alternatives holds:\begin{enumerate} \item
$Q\prec_M A\rtimes_\sigma \Gamma$, \item There exits $\iota\in I$ such that $Q\prec_M A\rtimes_\sigma(H\wr_I \Gamma_\iota)$.\end{enumerate}
\end{lem}

\begin{proof}

Working with amplifications we can assume without loosing any generality that $p=1$. Since $v$ is a nonzero partial isometry we have that
$0<\|v\|^2_2\le 1$. Let $w$ be an element in $\tilde{M}$ such that $\|v-w\|_2<\frac{\|v\|^2_2}{4}$ and $w=\sum_{\tilde{\eta} \in S, g\in
K}a_{\tilde{\eta} g}u_{\tilde{\eta} g}$ where $S\subset \tilde{H}^I$ and $K \subset\G$ are finite sets. Using the triangle inequality we have
that

\begin{equation}\label{46}\frac{\|v\|_2^2}{2} \le |\tau(\theta_1(u^*)wuw^*)|, \text{  for all } u\in\mathcal{U}(Q).
\end{equation}

In the remaining part of the proof we show that if we assume $Q\nprec_MA\rtimes_\sigma\G$ and $Q\nprec_MA\rtimes_\sigma(H\wr_I\G_\iota)$ for
all $\iota \in I$ then we can find a unitaries $u\in\mathcal{U}(Q)$ such that $|\tau(\theta_1(u^*)wuw^*)|$ is as small as we like. When this is
combined with (\ref{46}) we get that $v=0$ which is a contradiction.

Every unitary $u\in M=A \rtimes_{\sigma} (H\wr_I\G)$ can be written as $u=\sum_{\eta g\in H\wr_I\G}a_{\eta g} u_{\eta g}$, where $a_{\eta g}\in
A$, $\eta\in H^I$ and $g\in \G$. Using these formulas one can evaluate:

\begin{equation*}|\tau(\theta_1(u^*)wuw^*)|=\end{equation*}

\begin{equation*}=|\tau\big(\sum_{\tilde{\eta},\tilde{\gamma}\in S;l,m\in K}\sum_{\xi g,\zeta k \in H\wr_I\G}
u_{g^{-1}\theta(\xi^{-1})}a^*_{\xi g}a_{\tilde{\eta} l}u_{\tilde{\eta} l}a_{\zeta k} u_{\zeta
k}u_{m^{-1}{\tilde{\gamma}^{-1}}}a^*_{\tilde{\gamma} m})|.\end{equation*}

\begin{equation*}\leq\sum_{\stackrel{\tilde{\eta},\tilde{\gamma}\in S;l,m\in K}{\xi g,\zeta k \in H\wr_I\G}}|\tau(\sigma_{g^{-1}\theta(\xi^{-1})}(a^*_{\xi g}a_{\tilde{\eta} l})\sigma_{g^{-1}\theta(\xi^{-1}) \eta l}(a_{\zeta k})\sigma_{g^{-1}\theta(\xi^{-1}) \eta l\zeta k m^{-1}{\tilde{\gamma}^{-1}}}(a^*_{\tilde{\gamma}
m})u_{g^{-1}\theta(\xi^{-1}) \tilde{\eta} l\zeta k m^{-1}{\tilde{\gamma}^{-1}}})|
\end{equation*}

\begin{equation}\label{41}
=\sum_{\stackrel{\tilde{\eta},\tilde{\gamma}\in S;l,m\in K;\xi g,\zeta k \in H\wr_I\G}{g^{-1}\theta(\xi^{-1})\tilde{\eta} l\zeta k
m^{-1}\tilde{\gamma}^{-1}=e}}|\tau(\sigma_{g^{-1}\theta(\xi^{-1})}(a^*_{\xi g}a_{\tilde{\eta l}})\sigma_{g^{-1}\theta(\xi^{-1}) \tilde{\eta}
l}(a_{\zeta k})a^*_{\tilde{\gamma} m})|
\end{equation}
Applying Cauchy-Schwartz inequality, the above quantity is smaller than

\begin{equation}\label{43}\sum_{\stackrel{\tilde{\eta},\tilde{\gamma}\in S;l,m\in K;\xi g,\zeta k \in H\wr_I\G}{g^{-1}\theta(\xi^{-1})\tilde{\eta} l\zeta k
m^{-1}\tilde{\gamma}^{-1}=e}}\|\sigma_{g^{-1}\theta(\xi^{-1})}(a^*_{\xi g}a_{\tilde{\eta }l})\|_2\|\sigma_{g^{-1}\theta(\xi^{-1}) \tilde{\eta}
l}(a_{\zeta k})a^*_{\tilde{\gamma} m}\|_2
\end{equation}
Denoting by $C=\max_{\stackrel{\tilde{\eta},\tilde{\gamma}\in S}{l,m\in K}}\{\|a_{\tilde{\eta}
l}\|_{\infty}\|a^*_{\tilde{\gamma}m}\|_{\infty}\}$ then continuing in (\ref{43}) we have that
\begin{equation}\label{44}\leq C\sum_{\stackrel{\tilde{\eta},\tilde{\gamma}\in S}{l,m\in K}} (\sum_{\stackrel{\xi,\zeta \in H^I}{g^{-1}\theta(\xi^{-1})\tilde{\eta} l\zeta k
m^{-1}\tilde{\gamma}^{-1}=e}}\|a_{\xi g} \|_2\|a_{\zeta k}\|_2)
\end{equation}

A simple computation shows that the equation $g^{-1}\theta(\xi^{-1}) \tilde{\eta} l\zeta k m^{-1}{\tilde{\gamma}^{-1}}=e$ is equivalent to
\begin{equation*}
\begin{array}{c}
                                                \rho_{g^{-1}}(\theta(\xi^{-1}) \tilde{\eta} )\rho_{g^{-1}l}(\zeta)\rho_{g^{-1}lkm^{-1}}(\tilde{\gamma}^{-1})=e \\
                                                \text{and }\\
                                                g^{-1}lkm^{-1}=e,
                                            \end{array}
\end{equation*} which is further equivalent with

\begin{equation*}\begin{array}{c}\zeta  =  \rho_{l^{-1}}(\tilde{\eta}^{-1}\theta(\xi))\rho_{km^{-1}}(\tilde{\gamma}) \\
                \text{and}  \\
                g  = lkm^{-1}.\end{array}
\end{equation*}

Next we let $\tilde{\gamma}=\tilde{\gamma}_1\gamma_2$ and $\tilde{\eta}=\tilde{\eta}_1\eta_2$ with $\eta_2,\gamma_2\in H^I$ and
$\tilde{\eta}_1,\tilde{\gamma}_1$ being either trivial or having all entries reduced words ending with a nontrivial letter from $\mathbb{Z}$.
Therefore plugging in the previous equation we obtain \begin{equation}\label{51}
\zeta=\rho_{l^{-1}}(\eta_2^{-1})\rho_{l^{-1}}(\tilde{\eta}_1^{-1})\rho_{l^{-1}}(\theta(\xi))\rho_{km^{-1}}(\tilde{\gamma}_1)\rho_{km^{-1}}(\gamma_2)\end{equation}
and since $\zeta,\rho_{l^{-1}}(\eta_2^{-1}),\rho_{km^{-1}}(\gamma_2)\in H^I$ it follows that
\begin{equation}\label{42}\rho_{l^{-1}}(\tilde{\eta}_1^{-1})\rho_{l^{-1}}(\theta(\xi))\rho_{km^{-1}}(\tilde{\gamma}_1)\in H^I.\end{equation}

Also notice there exist finite subsets, $L_{l,m,\tilde{\eta},\tilde{\gamma}}\subset\G$ and $F_{l,m,\tilde{\eta},\tilde{\gamma}}\subset I$ such
that $l^{-1} supp(\tilde{\eta}_1) \cap km^{-1}supp(\tilde{\gamma}_1)=\varnothing$ for every $k\in \G\setminus
\mathcal{F}_{l,m,\tilde{\eta},\tilde{\gamma}}$, where we denoted by $\mathcal{F}_{l,m,\tilde{\eta},\tilde{\gamma}}=\bigcup_{ z\in
L_{l,m,\tilde{\eta},\tilde{\gamma}},t\in F_{l,m,\tilde{\eta},\tilde{\gamma}} }z \G_t$.

When we combine the above paragraph with the fact that the first letter of every entry of $\rho_{l^{-1}}(\tilde{\eta}_1^{-1})$ and the last
letter of every entry of $\rho_{km^{-1}}(\tilde{\gamma}_1)$ are nontrivial elements in $\mathbb{Z}$, then (\ref{42}) implies
that\begin{equation*}\rho_{l^{-1}}(\tilde{\eta}_1^{-1})\rho_{l^{-1}}(\theta(\xi))\rho_{km^{-1}}(\tilde{\gamma}_1)=e.\end{equation*}

\noindent Hence, we obtain that for all $l,m\in K$, $\tilde{\eta},\tilde{\gamma}\in S$, $k\in \G\setminus
\mathcal{F}_{l,m,\tilde{\eta},\tilde{\gamma}}$ one has
\begin{equation}\label{45}
\begin{array}{c}\theta(\xi)=\tilde{\eta}_1\rho_{lkm^{-1}}(\tilde{\gamma}_1^{-1})\text{; }
\zeta=\rho_{l^{-1}}(\eta^{-1}_2)\rho_{km^{-1}}(\gamma_2)\\
\text{ and } \\
\theta^{-1}(\tilde{\eta}_1),\theta^{-1}(\tilde{\gamma}_1)\in H^I.
\end{array}
\end{equation}

On the other hand, it can be easily seen that for every $l,m \in K; \tilde{\eta},\tilde{\gamma}\in S, z \in L_{l,m,
\tilde{\eta},\tilde{\gamma}}$ there exist $x_{l,m,\tilde{\eta},\tilde{\gamma},z},y_{l,m,\tilde{\eta},\tilde{\gamma},z}\in H^I$ (depending only
on $l,m,\tilde{\eta},\tilde{\gamma},z$ ) such that the equation (\ref{51}) is equivalent with
\begin{equation}\label{52}\zeta=x_{l,m,\tilde{\eta},\tilde{\gamma},z}\xi y_{l,m,\tilde{\eta},\tilde{\gamma},z}.\end{equation}

\noindent Hence (\ref{44}) is equal to

\begin{equation*}=C \sum_{\stackrel{\tilde{\eta},\tilde{\gamma}\in S}{l,m\in K}} (\sum_{\stackrel{\stackrel{\xi,\zeta \in H^I}{\zeta=\rho_{l^{-1}}(\tilde{\eta}^{-1}\theta(\xi))\rho_{km^{-1}}(\tilde{\gamma})}}{\stackrel{g=lkm^{-1}}{ k\in \Gamma\setminus \mathcal{F}_{l,m,\tilde{\eta},\tilde{\gamma}}}}}\|a_{\xi g}
\|_2\|a_{\zeta k}\|_2 + \sum_{\stackrel{\stackrel{\xi,\zeta \in
H^I}{\zeta=\rho_{l^{-1}}(\tilde{\eta}^{-1}\theta(\xi))\rho_{km^{-1}}(\tilde{\gamma})}}{\stackrel{g=lkm^{-1}}{ k\in
\mathcal{F}_{l,m,\tilde{\eta},\tilde{\gamma}}}}}\|a_{\xi g}\|_2\|a_{\zeta k}\|_2),
\end{equation*}

\noindent and using relations (\ref{45}) and (\ref{52}) this is furthermore equal to

\begin{equation*}=C \sum_{\stackrel{\tilde{\eta},\tilde{\gamma}\in S}{l,m\in K}}( \sum_{k\in \G\setminus \mathcal{F}_{l,m,\tilde{\eta},\tilde{\gamma}}}\|a_{\theta^{-1}(\tilde{\eta}_1)lkm^{-1}\theta^{-1}(\tilde{\gamma}_1^{-1})}
\|_2\|a_{\rho_{l^{-1}}(\eta_2^{-1})km^{-1}\gamma_2m}\|_2
\end{equation*}

\begin{equation*}+\sum_{k=z\alpha\in \mathcal{F}_{l,m,\tilde{\eta},\tilde{\gamma}}}\sum_{\xi \in H^I }\|a_{\xi lkm^{-1}}\|_2\|a_{x_{l,m,\tilde{\eta},\tilde{\gamma},z}\xi
y_{l,m,\tilde{\eta},\tilde{\gamma},k}z}\|_2).
\end{equation*}

Splitting up the sum of the second term and using the definition for elements in $\mathcal{F}_{l,m,\tilde{\eta},\tilde{\gamma}}$ we get

\begin{equation*}=C \sum_{\stackrel{\tilde{\eta},\tilde{\gamma}\in S}{l,m\in K}}( \sum_{k\in \G\setminus \mathcal{F}_{l,m,\tilde{\eta},\tilde{\gamma}}}\|a_{\theta^{-1}(\tilde{\eta}_1)lkm^{-1}\theta^{-1}(\tilde{\gamma}_1^{-1})}
\|_2\|a_{\rho_{l^{-1}}(\eta_2^{-1})km^{-1}\gamma_2m}\|_2
\end{equation*}

\begin{equation*}+\sum_{z\in L_{l,m,\tilde{\eta},\tilde{\gamma}}} \sum_{t\in
F_{l,m,\tilde{\eta},\tilde{\gamma}}} \sum_{\alpha\in \Gamma_t}\sum_{\xi \in H^I }\|a_{\xi
lkm^{-1}}\|_2\|a_{x_{l,m,\tilde{\eta},\tilde{\gamma},z}\xi y_{l,m,\tilde{\eta},\tilde{\gamma},z}k}\|_2)
\end{equation*}

Above note that $k=z\alpha$. Now by Cauchy-Schwarz inequality we have that

\begin{equation*}\leq C \sum_{\stackrel{\tilde{\eta},\tilde{\gamma}\in S}{l,m\in K}}( \sum_{k\in \G\setminus \mathcal{F}_{l,m,\tilde{\eta},\tilde{\gamma}}}\|a_{\theta^{-1}(\tilde{\gamma}_1)lkm^{-1}\theta^{-1}(\tilde{\gamma}_1^{-1})}
\|^2_2)^{\frac{1}{2}}(\sum_{k\in \G\setminus
\mathcal{F}_{l,m,\tilde{\eta},\tilde{\gamma}}}\|a_{\rho_{l^{-1}}(\eta_2^{-1})k\rho_{m^{-1}}(\gamma_2)}\|^2_2)^{\frac{1}{2}} +
\end{equation*}

\begin{equation*}+\sum_{\stackrel{z\in L_{l,m,\tilde{\eta},\tilde{\gamma}}}{t\in F_{l,m,\tilde{\eta},\tilde{\gamma}}}} (\sum_{\stackrel{\alpha\in \Gamma_t}{\xi \in H^I}}\|a_{\xi lkm^{-1}}\|^2_2)^{\frac{1}{2}}
(\sum_{\stackrel{\alpha\in \Gamma_t}{\xi \in H^I}}\|a_{x_{l,m,\tilde{\eta},\tilde{\gamma},z}\xi
y_{l,m,\tilde{\eta},\tilde{\gamma},z}k}\|^2_2)^{\frac{1}{2}})
\end{equation*}

\begin{equation*}\leq C \sum_{\stackrel{\tilde{\eta},\tilde{\gamma}\in S}{l,m\in K}}(\|E_{A\rtimes \G}(u_{\theta^{-1}(\tilde{\eta}_1^{-1})}uu_{\theta^{-1}(\tilde{\gamma}_1)})\|_2\|E_{A \rtimes \G}(u_{\rho_{l^{-1}}(\eta_2)}uu_{\rho_{m^{-1}}(\gamma_2)})\|_2
\end{equation*}

\begin{equation}\label{47}+\sum_{\stackrel{z\in L_{l,m,\tilde{\eta},\tilde{\gamma}}}{t\in F_{l,m,\tilde{\eta},\tilde{\gamma}}}}\|E_{A\rtimes (H\wr_I \G_t))}(uu_{m^{-1}z^{-1}l^{-1}})\|_2\|E_{A\rtimes (H\wr_I \Gamma_t)}(u_{x^{-1}_{l,m,\tilde{\eta},\tilde{\gamma},z}}uu_{k^{-1}y^{-1}_{l,m,\tilde{\eta},\tilde{\gamma},z}}u_z)\|_2)
\end{equation}

If one assumes that the conclusion does not hold, i.e. $Q\nprec_{M}A\rtimes H^{I}$ and $Q\nprec_{M}A\rtimes\Gamma$, then by Theorem
\ref{popaintertwining} there exist a unitaries $u\in \mathcal{U}(Q)$ such that for all $l,m,\tilde{\eta}, \tilde{\gamma}$ we can make the terms
in the above sums as small as we like. In particular, using inequality (\ref{47}) there exits a unitaries $u\in \mathcal{U}(Q)$ such that the
quantity $|\tau(\theta_1(u^*)wu w^*)|$ is arbitrary small thus proving our claim.\end{proof}

So for the proof of the main theorem, if $\theta_t\ra id$  converges
uniformly on the unit ball of $Q$, then by the above two lemmas we
have that either $Q\prec_M A\rtimes\Gamma$ or $Q\prec_M
A\rtimes(H\wr_I \G_{\iota})$ for some $\iota \in I$. In the second
case, we can view $Q$ as embedded in a corner of $A\rtimes(H\wr_I
\G_{\iota})$ and since $\theta_t$ converges uniformly to $id_Q$ then
an averaging argument shows that $\theta_t$ must be implemented by a
partial isometry $v$ in $\tilde M$. Looking closely, it would seem
that $v$ would have to conjugate each coordinate of $H^I$ by $u$,
since this is exactly what $\theta_t$ does. However, the only way
for this to happen would be if the algebra $Q$ would be supported on
$H^F$, for some finite set $F\subset I$. In fact we show below this
is indeed the case. In order to do so, and thus finish the proof of Theorem
\ref{intertwiningrigid}, we need the following lemma whose proof is a straightforward 
adaption of the proof of Theorem 3.6 (ii) in \cite{Io06}. We include full details though, 
for the reader's convenience.

\begin{lem}\label{controlrigidincore} Let $\tilde M$ and $M$ as in Theorem \ref{intertwiningrigid} and let $N\subset p(A\rtimes(H\wr_{I}\G_{\iota}))p$
such that $\theta_t\rightarrow id$ on $N$.  Then one can find a finite set $F\subset I$ such that $N\prec_M
A\rtimes(H\wr_{\G_{\iota}F}\G_{\iota})$.
\end{lem}
\begin{proof}

Notice that, since $\theta_t\rightarrow id$ uniformly on $(N)_1$ then  by Lemma \ref{rigidaltern} there exists $t>0$ and a nonzero partial
isometry $v\in \tilde{M}$ such that \begin{equation}\label{401}\theta_t(x)v=vx\text{ for all }x\in N.\end{equation}

Consider the Fourier expansion $v=\sum_{\tilde{\eta}g\in \tilde{H}\wr_I\Gamma} a_{\tilde{\eta}g}u_{\tilde{\eta}g}$ and letting
$v_g=\sum_{\tilde{\eta}\in\tilde{H}^I}a_{\tilde{\eta}g}u_{\tilde{\eta}}\in A\rtimes_\sigma \tilde{H}^I$ we have that
$v=\sum_{g\in\Gamma}v_gu_{g}$.

Fix $g\in\Gamma$ such that $v_g\neq0$ and denote by $n$ the cardinality of the stabilizing group $\G_\iota$.

 We know that given any $\epsilon>0$ and any $k\in \G_\iota g$ and $h\in g \G_\iota$ we can find a finite set $F\subset I$ and a finite collection $v_k', v_h'\in
A\rtimes \tilde{H}^F$, such that $\|v_k-v_k'\|_2<\frac{\epsilon}{2n}$, and $\|v_h-v_h'\|_2<\frac{\epsilon}{2n}$. If we let $x=\sum_{\gamma\in
\G_\iota}x_\gamma u_\gamma\in N $ then, identifying the $u_g$ coefficient on both sides of equation (\ref{401}), we have

\begin{eqnarray*}\sum_{\gamma k=g}\theta_t(x_\gamma)\sigma_\gamma(v_k)=\sum_{h\gamma=g}v_h\sigma_h(x_\gamma)\text{ for all }x\in N.\end{eqnarray*}

Combining this with the above inequalities we obtain

\begin{eqnarray}\label{402}\|\sum_{\gamma k=g}\theta_t(x_\gamma)\sigma_\gamma(v_k')-\sum_{h\gamma=g}v_h'\sigma_h(x_\gamma)\|_2<2\epsilon\text{ for all }x\in (N)_1.\end{eqnarray}

Since $\sum_{\gamma k=g}\theta_t(x_\gamma)\sigma_\gamma(v_k')\in\mathcal{H}=L^2(\theta_t(A\rtimes (H\wr_{I\setminus
\G_{\iota}F}\G_{\iota})))\overline{\otimes} L^2(A\rtimes(\tilde{H}\wr_{\G_{\iota}F}\G_{\iota}))$, if we let $T$ be the orthogonal projection
onto $\mathcal{H}$, then combining the above with triangle inequality we obtain

\begin{eqnarray*}\|T(\sum_{h\gamma=g}v_h'\sigma_h(x_\gamma))-\sum_{h\gamma=g}v_h'\sigma_h(x_\gamma)\|_2<4\epsilon\text{ for all }x\in (N)_1.\end{eqnarray*}

On the other hand for every $x\in L(H)$ we have $E_{\theta_t(L(H))}(x)=|\tau(u_t)|^2\theta_t(x)$ and therefore a little computation shows that
for all $\xi\in L^2(A\rtimes H\wr_{I\setminus \G_{\iota}F}\G_{\iota})\overline{\otimes}L^2(A\rtimes\tilde{H}\wr_{\G_{\iota}F}\G_{\iota})$ we
have

$$\|T(\xi)\|_2^2\le|\tau(u_t)^4|\|\xi\|_2^2+(1-|\tau(u_t)|^4)\|E_{A\rtimes\tilde{H}\wr_{\G_{\iota}F}\G_{\iota}}(\xi)\|_2^2.$$

Using the last inequality for $\xi=\sum_{h\gamma=g}v_h'\sigma_h(x_\gamma)$ in combination with (\ref{402}) we get that for all
$x\in\mathcal{U}(N)$ we have

\begin{eqnarray*}&&\|E_{A\rtimes(\tilde{H}\wr_{\G_{\iota}F}\G_{\iota})}(\sum_{h\gamma=g}v_h'\sigma_h(x_\gamma))\|_2^2\\
&\ge&(1-|\tau(u_t)|^4)^{-1}[(\|\sum_{h\gamma=g}v_h'\sigma_h(x_\gamma)\|_2-4\epsilon)^2-|\tau(u_t)|^4\|\sum_{h\gamma=g}v_h'\sigma_h(x_\gamma)\|^2_2]\\
&=&\|\sum_{h\gamma=g}v_h'\sigma_h(x_\gamma)\|^2_2- (1-|\tau(u_t)|^4)^{-1}(8\epsilon\|\sum_{h\gamma=g}v_h'\sigma_h(x_\gamma)\|_2-16\epsilon^2)\\
&\ge&(\|\sum_{h\gamma=g}v_h\sigma_h(x_\gamma)\|_2-\epsilon)^2-
(1-|\tau(u_t)|^4)^{-1}(8\|\sum_{h\gamma=g}v_h'\sigma_h(x_\gamma)\|_2\epsilon-16\epsilon^2).
\end{eqnarray*}

Choosing $\epsilon$ sufficiently small one can find a element $g\in\G$ and a constant $c>0$ such that for all $x\in \mathcal U(N)$ we have
\begin{eqnarray*}
\|E_{A\rtimes\tilde{H}\wr_{\G_{\iota}F}\G_{\iota}}(\theta_t(x)vu_{g}^*)\|_2&=&\|E_{A\rtimes\tilde{H}\wr_{\G_{\iota}F}\G_{\iota}}(vxu_{g}^*)\|_2\\
&=&\|E_{A\rtimes\tilde{H}\wr_{\G_{\iota}F}\G_{\iota}}(v_g\sigma_g(x))\|_2\ge c.
\end{eqnarray*}
 This implies
$\|E_{A\rtimes(\tilde{H}\wr_{\G_{\iota}F}\G_{\iota})}(x\theta_{-t}(v)u_{g}^*)\|_2\ge c$ and by expanding $F$ to a larger finite set if
necessary, we can find $v'\in A\rtimes(\tilde{H}\wr_{\G_{\iota}F}\G_{\iota})$ with $v'$ close to $v$ in $\|\cdot\|_2$ such that

\begin{equation*}\|E_{A\rtimes\tilde{H}\wr_{\G_{\iota}F}\G_{\iota}}(x\theta_{-t}(v')u_{g}^*)\|_2\ge \frac{c}{2}.\end{equation*}

Now if we further truncate $v'$ such that it is supported on elements of $\tilde{H}\wr_{\G_{\iota}F}\G_{\iota}$ with bounded world length in
$\tilde{H}$ then we can find elements $a_1,...,a_n\in M$ with

\begin{eqnarray*}\sum_i\|E_{A\rtimes(H\wr_{\G_{\iota}F}\G_{\iota})}(xa_i)\|_2\ge \frac{c}{4},\end{eqnarray*}

\noindent and therefore by Theorem \ref{popaintertwining} we have $N\prec_M A\rtimes (H\wr_{\G_{\iota}F}\G_{\iota})$.\end{proof}

Applying Theorem \ref{intertwiningrigid2} in the context of rigid, i.e. property (T), subalgebras of $M$ we obtain the following structural result

\begin{thm}\label{intertwiningrigid} Let $H,\G$ countable groups and let $I$ a $\G$-set with finite stabilizers.
Let $H\wr_{I}\Gamma\curvearrowright A$ be a trace preserving action on a finite algebra $A$ and denote by $M=A\rtimes( H\wr_{I}\Gamma)$. If
$Q\subset pMp$ is a diffuse rigid subalgebra then one of the following alternatives holds:
\begin{enumerate}
\item  $Q\prec_M A\rtimes\Gamma$,
\item There exists $\iota \in I$ and a finite set $F\subset I$ such that $Q\prec_M
A\rtimes (H\wr_{\G_{\iota}F} \G_{\iota})$.
\end{enumerate}
\end{thm}
\begin{proof}Since $Q\subset pMp$ is rigid, we know that $Q\subset p\tilde{M}p$
is rigid as well. Thus, we can find $F\subset\tilde{M}$ finite and $\delta > 0$ such that if $\phi : \tilde{M}\rightarrow\tilde{M}$ is a
normal, subunital, c.p. map with $\|\phi(x) - x\|_2\le\delta$, for all $x\in F$, then $\|\phi(u) - u\|_2\le\frac{1}{2}$, for all
$u\in\mathcal{U}(Q)$. In particular, since $t\rightarrow\theta_t$ is a pointwise $\|.\|_2$-continuous action, we can find $n\geq1$ such that
$\|\theta_{1/2^n}(u) - u\|\le\frac{1}{2}$, for all $u\in\mathcal{U}(Q)$. Therefore $\theta_t\ra id$ uniformly on the unit ball of $(Q)_1$ and
the conclusion follows from Theorem \ref{intertwiningrigid2}.\end{proof}

Also, for further use, we point out the following consequence of the above theorem:
\begin{thm}\label{controlrigid2}  Let $H$ be a group with Haagerup's property and $I$ a $\G$-set with finite stabilizers.
Let $H\wr_{I}\Gamma\curvearrowright A$ be a trace preserving action on an abelian algebra $A$ and denote by $M=A\rtimes( H\wr_{I}\Gamma)$. If
$Q\subset M$ is a diffuse property (T) subalgebra then $Q\prec_M A\rtimes\G$.
\end{thm}
\begin{proof}
Notice that, by Theorem \ref{intertwiningrigid}, we only need to show that $Q\nprec_M A\rtimes( H\wr_I \G_\iota)$. Below we proceed by
contradiction to show this is indeed the case.

So assuming $Q\nprec_M A\rtimes( H\wr_I \G_\iota)$, without loosing any generality, we may actually suppose that $Q\subset A\rtimes ( H\wr_I
\G_\iota)$ is a possibly non-unital subalgebra.

Since $H$ has Haagerup property it follows that $H\wr_I\G_\iota$ also has the Haagerup property. Therefore one can find a sequence,
$\{\phi_n\}\in c_o(H\wr_I\G_\iota)$, of positive definite functions that converge to the constant function 1 pointwise. It is well known that
the corresponding multipliers $m_n=m_{\phi_n}:A\rtimes (H\wr_I\G_\iota)\rightarrow A\rtimes (H\wr_I\G_\iota)$ given by $m_n(\sum a_gu_g)=\sum
\phi_n(g)a_gu_g$ form a sequence of completely positive maps converging pointwise to the identity. Since $Q$ has property (T), they must
converge uniformly on the unit ball of $Q$. Thus there is a finite set $F\subset H\wr_I\G_\iota$ such that if $x=\sum_{g\in
H\wr_I\G_\iota}x_gu_g\in (Q)_1$ then $\|\sum_{g\in F}x_gu_g\|_2>\frac{1}{2}$ for all $x\in (Q)_1$.

This implies that $\sum_{g\in F}\|E_A(xu_g^*)\|_2>\frac{1}{2}$ and by Theorem \ref{popaintertwining} we obtain $Q\prec_M A$, which is a
contradiction because $A$ is abelian while $Q$ has property (T).\end{proof}

\section{Commuting Subalgebras of M.}\label{sec:commutingsubalg}

 In this section we study commuting subalgebras of von Neumann algebras arising from actions by wreath product groups.
 Our main result is a general theorem describing the position of all subalgebras of $M$ having large commutant. The first result in this direction
 was obtained by the second named author in \cite{P-gap}, in the context of von Neumann algebras arising from Bernoulli actions. For similar results
 the reader may consult \cite{Oz04,CI08}.

\vskip 0.1in
\begin{thm}\label{comint} Let $H,\G$ be countable groups with $H$ amenable and let $I$ a $\G$-set with finite stabilizers.
Let $H\wr_{I}\Gamma\curvearrowright A$ be a trace preserving action on an amenable algebra $A$ and denote by $M=A\rtimes( H\wr_{I}\Gamma)$. Let
$p\in M$ be a projection and  $P\subset pMp$ be a subalgebra with no amenable direct summand. If we denote by $Q=P'\cap pMp$ then we have
that $Q\prec_M A\rtimes\Gamma$.

Moreover, if we also assume that $A\rtimes\G$ is a factor and  $Q\nprec_M A\rtimes\G_{\iota}$ for all $\iota$ then there exists a unitary $u\in M$
such that $u^*\mathcal{N}_M(Q)''u\subseteq A\rtimes \G$.

\end{thm}

Our proof is again based on deformation/rigidity technology,
resembling the proof of Theorem \ref{intertwiningrigid}. The main
difference however is that, instead of property (T), we will use the
``spectral gap rigidity'' argument from \cite{P-gap} to show that
the deformation $\theta_t$ converges uniformly to the identity on
the unit ball of $Q$.  For the proof of Theorem \ref{comint} we need
the following preliminary result.
\begin{lem}\label{bimoduledecomp}Let $M$ and  $\tilde{M}$ as above and let $\omega$ be a free ultrafilter on $\mathbb N$. If $P\subset M\subset \tilde{M}$ is a subalgebra with no amenable direct summand then $P'\cap \tilde{M}^{\omega}\subset
M^{\omega}$.
\end{lem}

\begin{proof} The first step is to decompose the $M$-bimodule $L^2(\tilde{M})\ominus L^2(M)$ as a direct sum of cyclic
$M$-bimodules. It is a straightforward exercise for the reader to see that the above $M$-bimodule can be written as a direct sum of
$M$-bimodules $\overline{M\tilde{\eta}_sM}^{\|\cdot\|_2}$, where the cyclic vectors $\tilde{\eta}_s$ correspond to an enumeration of all
elements of $\tilde{H}^I$ whose non-trivial coordinates start and end with non-zero powers of $u$.

Next, for every $s$, we denote by $\eta_s$ the element of $H^I$ that
remains from $\tilde{\eta}_s$ after deleting all nontrivial powers
of $u$. Also for every $s$ let $\Delta_s$ be the support of
$\tilde{\eta}_s$ in $I$ and observe that if
$Stab_{\Gamma}(\tilde{\eta}_s)$ denotes the stabilizing group of
$\tilde{\eta}_s$ inside $\G$ then we have
$Stab_{\Gamma}(\tilde{\eta}_s)(I\setminus \Delta_s)\subset
(I\setminus \Delta_s)$. Hence we can consider the von Neumann
algebra $K_s=A\rtimes_{\sigma }(H\wr_{I\setminus
\Delta_s}Stab_{\Gamma}(\tilde{\eta}_s)) $ and using similar
computations as in Lemma 5 of \cite{CI08}, one can easily check that
the map $x\tilde{\eta}_sy\ra x\eta_se_{K_s}y$ implements an
$M$-bimodule isomorphism between $\overline {M\tilde{\eta}_s
M}^{\|\cdot\|_2}$ and
$L^2(\langle M,e_{K_s }\rangle)$. 

Therefore, as $M$-bimodules, we have the following isomorphism
\begin{equation}\label{400}L^2(\tilde{M})\ominus L^2(M)\cong \bigoplus_sL^2(\langle M,e_{K_s}\rangle).\end{equation}
Notice that, since $I$ is a $\Gamma$-set with amenable, in fact
finite, stabilizers if follows that $Stab_{\G}(\tilde{\eta}_s)$ are
amenable for all $s$. Also, since $H$ is amenable group and $A$ is
an amenable algebra we conclude that the algebra $K_s$ is amenable
for all $s$ and therefore the bimodule in (\ref{400}) is weakly
contained in a multiple of the coarse bimodule
$L^2(M)\overline{\otimes} L^2(M)$. Finally, the conclusion of our
lemma follows proceeding exactly as in Lemma 5.1 from \cite{P-gap}.
\end{proof}

We can now proceed with the proof of Theorem \ref{comint}.

\begin{proof}
First we use the spectral gap argument to show that the deformation
$\theta$ converges to the identity uniformly on $(Q)_1$ .Indeed,
exactly as in \cite{P-gap}, since $P$ has no amenable direct
summand, Lemma \ref{bimoduledecomp} implies that $P'\cap
\tilde{M}^{\omega}\subset M^{\omega}$. Hence, for any $\epsilon>0$
there exist $\delta_{\epsilon}>0$ and  $\mathcal F\in
\mathcal{U}(P)$ a finite set, such that whenever $x\in \tilde{M}$
satisfies $\|[x,u]\|_2\le \delta_{\epsilon}$ for all $u\in\mathcal
F$ we have that $\|x-E_{{M}}(x)\|_2\le \epsilon$.

If we let $t_\epsilon>0$ such that $\|\theta_t(u)-u\|\le \frac{\delta_{\epsilon}}{2}$ for all $u\in \mathcal F$ then the triangle inequality
implies that for every $0\le t\le t_{\epsilon}$ and every $x\in (Q)_1$ we have
\begin{equation*}\|[\theta_t(x),u]\|_2\le 2\|\theta_t(u)-u\|\le \epsilon.\end{equation*} Therefore by the above we obtain that $\|\theta_t(x)-E_M(\theta_t(x))\|_2\le \epsilon$
 and using the transversality of $\theta_t$ (Theorem \ref{transversality}) we conclude that $\|\theta_{2t}(x)-x\|_2\le2\epsilon$ for all $x\in (Q)_1$ and $0\le t\le t_{\epsilon}$.

In conclusion deformation $\theta_t$ converges uniformly on $(Q)_1$ and hence, by applying Theorem \ref{intertwiningrigid2}, we have the
following two alternatives: either $Q\prec_M A\rtimes \G $ or there exist $\iota\in I$ and a finite set $F$ such that $Q\prec_M A\rtimes
(H\wr_{{\G}_\iota F} {\G}_\iota)$.

Next we show that the second case, together with the assumption
$Q\nprec_M A\rtimes_{\sigma} \G_j$ for all $j\in I $ will lead to a
contradiction. By these assumptions, using \cite{V-bim}, one can
find nonzero projections $q\in Q$, $p\in A\rtimes(H\wr_{\G_{\iota}F}
\G_{\iota})$, a $*$-homomorphism $\phi :qQq\ra
p(A\rtimes(H\wr_{\G_{\iota}F} \G_{\iota})) p$ and a partial isometry
$w\in M$ such that $\phi(x)w=wx$ for all $x\in qQq$ and
$\phi(qQq)\nprec_M A\rtimes \G_j$ for all $j\in I$. \newline Since
$\phi (qQq)$ is a diffuse subalgebra of $
p(A\rtimes(H\wr_{\G_{\iota}F} \G_{\iota})) p$ then part (3) of Lemma
\ref{controlnorm} implies that
\begin{equation}\label{100}\phi (qQq)'\cap pMp\subset \sum_{s\in\tilde{K}} [A\rtimes(H\wr_{I} \G_{\iota})] u_s.\end{equation} On
the other hand $P\subset Q'\cap M $ and hence by (\ref{100}) we have $wPw^*\subset \sum_{s\in\tilde{K}}[A\rtimes(H\wr_{I} \G_{\iota})]u_s $.
  Since $\tilde{K}=\bigcup_{k,l\in F} g_{k,l}\G_{\iota} \G_k$ for some finite set of elements $g_{k,l}\in \G$ then by above we have
  that $wPw^*\subset \sum_{k,l}[A\rtimes(H\wr_{I} \G_{\iota}\G_k)]u_{g_{k,l}} $. Using intertwining by bimodule techniques this implies
 that $P\prec_M A\rtimes(H\wr_{I} \G_{\iota}\G_{k_o})$ for some $k_o\in F$ but this is impossible because $A\rtimes(H\wr_{I} \G_{\iota}Stab_{\G}(k_o))$ is amenable
while $P$ has no amenable direct summand.

Therefore the only possibility is $Q\prec_M A\rtimes \G $ and the
remaining part of the conclusion follows proceeding in the same way
as in Theorem 4.4 \emph{ii)} of \cite{P1}.
\end{proof}

An algebra $N$ is called \emph{solid} if for every $A\subset N$
diffuse subalgebra $A'\cap N$ is amenable. As a consequence of
previous theorem we obtain the following stability property similar
with Corollary 8 in \cite{CI08}.

\begin{cor} Let $(A,\tau)$ be an amenable von Neumann
algebra and $H$ be an amenable group. Assume that $(H\wr\Gamma)\curvearrowright A$ is a trace preserving action such that $M=A\rtimes
(H\wr\Gamma)$ and $A\rtimes\Gamma$ are factors and for every diffuse $Q\subset A$ the relative commutant $Q'\cap M$ is amenable. Then $A\rtimes
(H\wr\Gamma)$ is a solid if and only if $A\rtimes\Gamma$ is solid.
\end{cor}

\begin{proof}  Notice that the proof follows once we show that $A\rtimes\Gamma$ is solid implies $A\rtimes (H\wr\Gamma)$ is a solid.
Hence assume that $A\rtimes\Gamma$ is solid and let $B\subset M=A\rtimes (H\wr\Gamma)$ be a diffuse von Neumann subalgebra. If we assume by
contradiction that the commutant $P = B'\cap M$ is non-amenable, then we can find a non-zero projection $z\in\mathcal{Z}(P)$ such that $Pz$ has
no amenable direct summand. Since $[Bz,Pz] = 0$ then $Bz\prec_MA\rtimes\G$ and by the hypothesis assumption we have that $Bz\nprec_MA$.
Therefore, since $A\rtimes\G$ is a factor then by the second part of Theorem \ref{comint} one can find a unitary $u\in M$ such that $u(Bz\vee
Pz)u^*\subset A\rtimes \G$. This however contradicts the solidity of $A\rtimes \G$ and we are done.

\end{proof}
\begin{rem} It is immediate from Theorem \ref{comint} that if $H$ is an amenable group then for any
non-amenable group $\G$ and any free, ergodic, measure
preserving action $H\wr\Gamma\curvearrowright(X,\mu)$ the II$_1$ factor $L^{\infty}(X,\mu)\rtimes (H\wr\Gamma)$ is prime, i.e. it cannot be
decomposed as a tensor product of two diffuse factors.
\end{rem}

\bigskip

\section{OE-rigidity results}\label{sec:measure equivalence}

Sako showed in \cite{Sa09} that a measure equivalence between two
wreath products groups $H\wr\Gamma$ and $K\wr\La$,
 where $H,K$ are amenable and $\Gamma,\La$ are products of non-amenable exact groups, implies the measure equivalence of the malnormal subgroups $\Gamma$ and $\La$.
 In fact he was able to prove this measure equivalence rigidity for certain classes of direct products and amalgamated free products,
 thus obtaining rigidity results \'{a} la  Monod-Shalom \cite{MoSh}, as well as of Bass-Serre type \cite{IPP,AG08,C-H}.
 His methods rely on Ozawa's techniques \cite{Oz,Oz04} involving the class $\mathcal{S}$ of groups, being
 $C^*$-algebraic in nature and depending crucially on exactness of the groups involved.

 In this section we apply the results from the previous section to show that this type of measure equivalence rigidity for wreath
 products holds true for much larger classes of groups (Corollary \ref{merigid2} below).
 The techniques we use in the proof
 are purely von Neumann algebra, using Popa's deformation/rigidity theory.

\vskip 0.05in

 {{\bf {The Classes WR}}(k).} Recall from the introduction that for each $k=1, 2, 3$, we
 denote by ${\bf WR}(k)$ the class of all generalized wreath product groups $H\wr_I\Gamma$
 with $\G$ i.c.c., $I$ a $\G$-set with finite stabilizers and satisfying the corresponding condition from below:

\begin{enumerate}\item \label{ht} $\Gamma$ has property (T) and $H$ has Haagerup's property;
 \item \label{tt}$\Gamma$ and $H$ have property (T) and $H$ is i.c.c.;
  \item \label{ap}$\Gamma$ is a non-amenable product of infinite groups and $H$ is amenable.

\end{enumerate}

\begin{thm}\label{merigid1} Let $H\wr_{I}\Gamma, K\wr_J \La\in {\bf WR}(k)$ and suppose that
$(H\wr_{I}\Gamma)\curvearrowright A$ and $(K\wr_{J}\La)\curvearrowright B$ are free, trace preserving actions on diffuse, abelian algebras.
Denote by $M=A\rtimes(H\wr_{I}\Gamma)$, $N=B\rtimes(K\wr_{J}\La)$, let  $t>0$ and  assume that $\phi: M\rightarrow N^t$ is a $*$-isomorphism
such that $\phi(A)=B^t$.

Then one can find a unitary $u\in \mathcal{N}_{N^t}(B^t)$ such that $u^*\phi(A\rtimes \Gamma)u=(B\rtimes \La)^t$.

\end{thm}

\begin{proof}

Denote by $P=A\rtimes \Gamma$, $Q=B\rtimes \La$ and observe that $A\subset P \subset M$ and $B\subset Q \subset N$. To simplify the
technicalities we will assume without loosing any generality that $t=1$. Since $\Gamma$ either has property (T) or is a non-amenable product of
infinite groups and $\phi$ is an isomorphism it follows that either $\phi(L\Gamma)$ is a property (T) subalgebra of $M$ or $\phi(L\Gamma)$ is a
non-amenable tensor product of two diffuse factors.

Below, we argue that for all cases (1)-(3) covered in the definition of the classes ${\bf WR}(k)$ we have
\begin{eqnarray}\label{200}\phi(L\G)\prec_{N}Q.\end{eqnarray}

For case (\ref{ht}) this follows directly from Corollary \ref{controlrigid2} while for case (\ref{ap}) it follows from Theorem \ref{comint}.
Therefore it only remains to treat case (\ref{tt}), i.e. when all groups $H, K,\G, \La$ have property (T).

Applying Theorem \ref{intertwiningrigid} we have that either $\phi (L\G)\prec_{N}Q$ or there exists a finite subset $T\subset J$ such that
$\phi(L\G)\prec_{N}B\rtimes K^{T}$ and therefore to finish the proof of (\ref{200}) it suffices to show that the second possibility leads to a
contradiction.

Notice that since  $\phi^{-1}(LK^{T})$ is a property (T) subalgebra of $M$ then Theorem \ref{intertwiningrigid} again implies that either
$\phi^{-1} (L K^{T})\prec_{M}P$ or there exists a finite subset $S\subset I$ such that $\phi^{-1} (L K^{T})\prec_{M}A\rtimes H^{S}$. Next we
show that both situations are leading to a contradiction.

Assuming the first situation, since $LK^T$ and $P$ are a factors,
then proceeding as in the proof of Theorem 5.1 in \cite{IPP} one can
find a nonzero projection $p_1\in LK^{J\setminus T}$ and a unitary
$u_1\in M$ such that  $u_1^*(\phi^{-1}((LK^T) p_1))u_1\subset P$.
Using Lemma \ref{controlnorm}, this implies that
$u_1^*(\phi^{-1}(p_1(LK^{J})p_1))u_1\subset P $. Moreover, since $P$
is a factor, we have that $u_1^*(\phi^{-1}(L(K^{J}))u_1\subset P $
and therefore Lemma \ref{controlnorm} implies that
$u_1^*(\phi^{-1}(L(K\wr_{J}\La))u_1\subset P $. However, since
$\phi^{-1}(B)=A$ then by Lemma \ref{augumentintertwining} again we
have that $M=\phi^{-1}(N)\prec_{M}P$, which is obviously a
contradiction.

Assuming the second situation, since $\phi^{-1}(B)=A$, then Lemma \ref{augumentintertwining} gives that $\phi^{-1}(B\rtimes
K^{T})\prec_{M}A\rtimes H^{S}$. From the initial assumptions $B\rtimes K^{T}$ is a factor and therefore Lemma \ref{intertwining} implies that
$\phi^{-1}(B\rtimes K^{T})^{\omega}\subset (A\rtimes H^{S})^\omega\vee M$ or equivalently
\begin{equation}\label{201}(B\rtimes K^{T})^{\omega}\subset (\phi(A\rtimes H^{S}))^\omega\vee N.\end{equation} Also, since
$\phi(L\G)\prec_{N}B\rtimes K^{T}$, the same argument as above shows that \begin{equation*}(\phi(L\G))^\omega\subset (B\rtimes
K^{T})^{\omega}\vee N,\end{equation*} and  combining this with (\ref{201}) we obtain that $(\phi(L\G))^\omega\subset(\phi(A\rtimes H^{I
}))^\omega\vee N$. Therefore the second part of Lemma \ref{intertwining} implies $L\G\prec_{M}A\rtimes H^{I}$ but one can easily see this is
again impossible. \vskip 0.05in

Hence we proved (\ref{200}) and, moreover, since $\phi(A)=B$ then Lemma \ref{augumentintertwining} implies that
\begin{equation}\label{203}\phi(P)\prec_{N}Q.\end{equation}

 Next we show that the intertwining above can be extended to unitary conjugacy preserving the Cartan subalgebra $B$.

 By (\ref{203}) one can find nonzero projections $p\in P$, $q\in Q$, a nonzero partial isometry $w\in M$ and a unital isomorphism $\psi:
\phi(pPp)\rightarrow qQq$ such that
\begin{equation}\label{7}w\psi(x)=xw \text{ for all } x\in \phi(pPp).\end{equation}
The previous relation automatically implies that $ww^*\in\phi(pPp)'\cap \phi(p)N\phi(p)$ and $w^*w\in\psi(\phi(pPp))'\cap qMq$. Since $P$ is a
factor then Lemma \ref{controlnorm} gives that $\phi(pPp)'\cap\phi(p)N\phi(p)=\mathbb C \phi(p)$ and therefore $ww^*=\phi(p)$.

Similarly, since $\psi(\phi(pPp))$ is a II$_1$ factor and $B\rtimes Stab_{\La}(j)$ is a type I algebra for all $j\in J$ then
$\psi(\phi(pPp))\nprec_{Q}B\rtimes Stab_{\La}(j)$ and by Lemma \ref{controlnorm} we have that $\psi(\phi(pPp))'\cap qNq\subset Q$. When this is
combined with the above we obtain $w^*w\in Q$ and hence relation (\ref{7}) implies that
\begin{equation}\label{8}
w^*\phi(P)w=w^*w\psi(\phi(pPp))\subseteq Q.
\end{equation}
Letting $v_0\in N$ to be a unitary such that $w=ww^*v_0$, the previous relation rewrites as  $v_0^*\phi(pPp)v_0\subseteq N_2$ and since $Q$ is
a factor one can find a unitary $v\in N$ such that
\begin{equation}\label{9} v\phi(P)v^*\subseteq Q.\end{equation}

Next we claim that $vBv^*\prec_{Q} B$. To see this, suppose by contradiction that $vBv^*\nprec_{Q} B$. Since $Stab_{\La}(j)$ is finite for all
$j\in J$ this is equivalent to $vBv^*\nprec_{Q} B\rtimes Stab_{\La}(j)$. Therefore Lemma \ref{controlnorm} implies that
$\mathcal{N}_{N}(vBv^*)''\subseteq Q$ and because $vBv^*$ is a Cartan subalgebra of $N$ one gets that $N\subset Q$. However this is impossible
and hence we proved our claim.

Furthermore, since $vBv^*$ and $B$ are Cartan subalgebras of $Q$
satisfying $vBv^*\prec_{Q} B$, Theorem A.1. in \cite{PBetti} shows
that there exists a unitary $v_1\in Q$ such that $v_1vBv^*v_1^*=B$.
Therefore $u=v_1v\in \mathcal{N}_{N}(B)$ and combining this with
(\ref{9}) we obtain that
\begin{equation}\label{10}u\phi(P)u^*\subseteq Q.\end{equation}
In the remaining part of the proof we show that the two algebras above coincide. Indeed, applying the same reasoning as before for the
isomorphism $\phi^{-1}$, one can find a unitary $u_o\in \mathcal{N}_{M}(A)$ such that
\begin{equation*}u_o\phi^{-1}(Q)u_o^*\subseteq P,\end{equation*}
and combining this with (\ref{10}) we obtain
\begin{equation}\label{12}
u_o\phi^{-1}(u)P\phi^{-1}(u^*)u_o^*\subseteq u_o\phi^{-1}(Q)u_o^*\subseteq P.
\end{equation}
However, Lemma \ref{controlnorm} implies that $u_o\phi^{-1}(u)\in P$ and therefore relation (\ref{12}) became
$u_o\phi^{-1}(u)P\phi^{-1}(u^*)u_o^*=u_o\phi^{-1}(Q)u_o^*=P$, which in particular entails that $u\phi(P)u^*= Q$.\end{proof}

\begin{thm}\label{merigid4} Let $H\wr_{I}\Gamma$,  $K\wr_J \La$ be generalized wreath product groups
such that $H$, $K$ are i.c.c.\ groups with  property $(T)$ and $I$, $J$ have finite stabilizers. Suppose that $(H\wr_{I}\Gamma)\curvearrowright
A$ and $(K\wr_{J}\La)\curvearrowright^\rho B$ are free, trace preserving actions on diffuse, abelian algebras and denote by
$M=A\rtimes(H\wr_{I}\Gamma)$, $N=B\rtimes(K\wr_{J}\La)$.

If $t>0$ and $\phi: M\rightarrow N^t$ is a $*$-isomorphism such that $\phi(A)=B^t$ then one can find a unitary $x\in \mathcal{N}_{N^t}(B^t)$
such that $x\phi(A\rtimes H^I)x^*=(B\rtimes K^J)^t$.

\end{thm}
\begin{proof} To simplify the technicalities we assume that $t=1$.
Since $H$ has property (T) then $\phi(LH)$ is a rigid subalgebra of $N$ and therefore by Theorem \ref{intertwiningrigid} we have that either
$\phi(LH)\prec_NB\rtimes \La$ or there exits a finite subset $T\subset J$ such that $\phi(LH)\prec_NB\rtimes K^T$. Using the same arguments as
in the first part of the proof one can easily show that the first possibility will lead to a contradiction. Therefore we have that
$\phi(LH)\prec_NB\rtimes K^T$ and by applying Lemma \ref{controlnorm} we get that $\phi(LH^I)\prec_NB\rtimes K^J$. Applying Lemma
\ref{controlnorm} this further implies that $\phi(A\rtimes H^I)\prec_NB\rtimes K^J$ and therefore there exists a $A\rtimes H^I$-$B\rtimes K^J $
bimodule $\mathcal H$ with finite dimension over $B\rtimes K^J$.

A similar argument for $\phi^{-1}$ shows that $B\rtimes
K^J\prec_N\phi (A\rtimes H^I)$ and hence one can find a nonzero
$B\rtimes K^J $-$A\rtimes H^I$ bimodule $\mathcal K$ with finite
dimension over $B\rtimes K^J$. Since $\G,\La$ are i.c.c.\ and
$B\rtimes K^J$ and $\phi(A\rtimes \G)$ are irreducible, regular
subfactors of $N$ then, by Theorem 8.4 in \cite{IPP}, there exists a
unitary $u\in N$ such that $u \phi (A\rtimes H^I) u^*=B\rtimes K^J$.
Denoting by $\psi_u=Ad(u)$ this further implies that
$\psi_u\circ\phi$ is an isomorphism from $A\rtimes H^I$ onto
$B\rtimes K^J$  which satisfies
\begin{equation*}\psi_u\circ \phi(a)u=u\phi(a),\end{equation*}
for all $a\in A$.
Next we consider the Fourier decomposition $u=\sum_{\lambda\in \La} y_\lambda v_\lambda$ with $y_\lambda \in B\rtimes K^J$ and using the above equation there exists a nonzero element $y_\lambda\in B\rtimes K^J$ such that for all $a\in A$ we have
\begin{equation}\label{700}\psi_u\circ \phi(a)y_\lambda=y_\lambda \rho_\lambda (\phi(a)).\end{equation}

Note that since $B=\phi(A)$ is a maximal abelian subalgebra of $N$ then (\ref{700}) implies that $y^*_\lambda y_\lambda \in B$. Furthermore taking the polar decomposition $y_\lambda =w_\lambda|y_\lambda|$ with $w_\lambda$ partial isometry in (\ref{700}) we conclude that

\begin{equation*}\psi_u\circ \phi(a)w_\lambda=w_\lambda \rho_\lambda (\phi(a)),\end{equation*}
for all $a\in A$.

This shows in particular $\psi_u (B)\prec_{B\rtimes K^J} B$ and
since $B$ and $\psi_u(B)$ are Cartan subalgebras of $B\rtimes K^J$
then by Theorem A.1 \cite{PBetti} there exists a unitary $u_o\in
B\rtimes K^J$ such that $u_o\psi_u(B)u^*_o=B$. Finally the
conclusion follows by letting $x=u_ou\in \mathcal N_N(B)$.
\end{proof}

We now
have the following immediate corollary of Theorem \ref{merigid1}:
\begin{cor}\label{merigid2}
Given $1\le k\le 3$ let $H\wr_{I}\G,K\wr_J\La \in{\bf WR}(k)$. If one assumes that $H\wr_{I}\Gamma\cong_{ME}K\wr_{J}\La$ then we have
$\Gamma\cong_{ME}\La$.
\end{cor}

 A natural question one may ask is to try classifying \emph{all} groups $\G$ and $H$ for which the above measure equivalence rigidity
phenomena holds. This however remains widely open as for the moment it is unclear what general condition one may be formulate at the level of
groups $\G$ and $H$ to insure this type of rigidity. For instance even when assuming $\G$ has property (T) it is not obvious what are all
groups $H$ for which this rigidity holds.

Another interesting problem is to find situations when measure equivalence rigidity can be upgraded also at the level of the ``core'' groups
$H^I$ and $K^J$. A desirable result in this direction would be that a measure equivalence between $H\wr\Gamma$ and $K\wr\La$ induces a measure
equivalence not only between the malnormal groups $\G$ and $\La$ but also between the normal groups $H^\G$ and $K^\La$. Notice that combining
Theorems \ref{merigid2} and \ref{merigid4} above we obtain one instance of this phenomenon.

\begin{cor}\label{merigid3} If
$H\wr_{I}\G,K\wr_J\La \in{\bf WR}(2)$ such that $H\wr_{I}\Gamma\cong_{ME}K\wr_{J}\La$ then we have $\Gamma\cong_{ME}\La$ and
$H^{I}\cong_{ME}K^{J}$.
\end{cor}


\section{W$^*$-rigidity results}\label{sec:further results}

Some of the technical results obtained in the previous sections can be pushed to slightly more general situations. For instance rather
than studying commuting subalgebras algebras of von Neumann algebras arising from actions by wreath product groups one can study \emph{weakly
compact embeddings.}

 This notion was introduced by Ozawa and Popa and it was triggered by their discovery that in a free group factor $M$ the normalizing
 group $\mathcal N_M(P)$ of any amenable algebra $P$ acts on $P$ by
conjugation in a ``compact'' way \cite{OP1}. This was a key
ingredient which allowed the authors to prove that in a free group
factor the normalizing algebra of any amenable subalgebra is still
amenable. For reader's conveninece, we recall the following
definition from \cite{OP1}:

\begin{defn}\label{wcdef} Let $\Lambda
\stackrel{\sigma}{\curvearrowright}P$ where $P$ is a finite von Neumann algebra. The action $\sigma$ is called \emph{weakly compact} if there
exist a net $(\eta_\alpha)$ of unit vectors in $L^2(P\bar{\otimes}\bar{P})_+$ such that:
\begin{eqnarray}
&&\|\eta_\alpha - (v\otimes\bar{v})\eta_\alpha\|_2\rightarrow 0 \qquad \text{for all } v\in \mathcal{U}(P);\\
&&\|\eta_\alpha - \sigma_g\otimes\bar{\sigma}_g(\eta_\alpha)\|_2\rightarrow 0 \qquad \text {for all } g\in\Gamma;\\
&&\langle(x\otimes1)\eta_\alpha, \eta_\alpha\rangle = \tau(x) = \langle\eta_\alpha, (1\otimes\bar{x})\eta_\alpha\rangle \qquad \text{for all }
\alpha \text{ and  }x\in P.
\end{eqnarray}
\end{defn}

If $P\subset M$ is a subalgebra such that the action
by conjugation of the normalizing group $\mathcal{N}_M(P)$ on $ P$
is weakly compact then we say that the inclusion $P\subset M$ is a
\emph{weakly compact embedding.} It is straightforward from the
definitions that every compact action action $\Lambda
\stackrel{\sigma}{\curvearrowright}P$ is automatically weakly
compact and hence every profinite action \cite{Io08} is also weakly
compact.

In the main result of this section we describe all weakly compact
embeddings in cross-products algebras of type $M= A \rtimes
(H\wr\G)$ with $A$ amenable algebra and $H$ amenable group. Roughly
speaking, we obtain a dichotomy result asserting that every weakly
compact embedding in $M$, either has ``small'' normalizing algebra
or ``lives'' inside $A\rtimes\G$. This should be seen as analogous
to Theorem 4.9 in \cite{OP1}. In fact our proof follows the same
recipe as the proof of Theorem 4.9 in \cite{OP1}. The main
difference at the technical level is that instead of working with
the malleable deformation for actions of free groups we will work
with the deformation described in the first section. Therefore the
compactness argument used in the proof of Theorem 4.9 in \cite{OP1}
will be replaced by the transversality property from Theorem
\ref{transversality}. Most of the arguments used in \cite{OP1} apply
verbatim in our situation and we include some details only for
reader's convenience.

\begin{thm}\label{wcint} Let $(A,\tau)$ be an amenable von Neumann algebras
and $H$ be an amenable group. Assume that $H\wr\Gamma\curvearrowright A$ is an trace preserving action and denote by $M = A \rtimes (H\wr\G)$.
If we assume that $P\subset M$ is a (diffuse) weakly compact embedding such that $\mathcal N_M(P)'\cap M =\mathbb C1$ then one of the
following must hold true:
\begin{enumerate}
\item There exists a nonzero projection $p\in P$ such that $\mathcal{N}_{pMp}(pPp)''$ is amenable.
\item $P\prec_{M}A\rtimes \Gamma$.
\end{enumerate}
If we assume in addition that $P\subset M$ is a Cartan subalgebra then we have that $P\prec_M A$. \end{thm}

\begin{proof}
Let $\mathcal{G}\subset\mathcal{N}_M(P)$ be a subgroup that acts weakly compactly on $P$ and assume that $\mathcal{U}(P)\subset\mathcal{G}$.
First we will show that, when we view $P\subset\tilde{M}$, if $\theta_t$ does not converge uniformly on $(P)_1$ then $\mathcal{G}''$ is
amenable.

So let us assume that $\theta_t$ does not converge uniformly on $(P)_1$. Therefore by transversality of $\theta_t$, Theorem
\ref{transversality}, one can find a constant $0<c<1$, and infinite sequences $t_k\in\mathbb{R}$, $u_k\in\mathcal{U}(P)$ such that
$t_k\rightarrow 0$ and

\begin{equation*}\|\theta_{t_k}(u_k)-E_M(\theta_{t_k}(u_k))\|_2\ge c.\end{equation*}

\noindent Since $\|\theta_{t_k}(u_k)\|_2=1$ then Pythagorean theorem further implies that

\begin{equation}\label{502}\|E_M(\theta_{t_k}(u_k))\|_2\le\sqrt{1-c^2}.\end{equation}

Now we fix  $\epsilon>0$ and $F\subset\mathcal{G}$ a finite set. Then we choose $\delta>0$ satisfying $1-2\delta>\sqrt{1-c^2}$ and $k$
sufficiently large such that for all $u\in F$ we have

\begin{equation*}\|u-\theta_{t_k}(u)\|\le \frac{\epsilon}{6}.\end{equation*}

For the rest of the proof we denote by $\theta=\theta_{t_k}$ and $v=u_k$ and let $(\eta_\alpha)$ be as in the definition of weak compactness.
Then we consider the following nets
\begin{eqnarray*}
&&\tilde{\eta}_\alpha=(\theta\otimes1)(\eta_\alpha)\in L^2(\tilde{M})\overline{\otimes}L^2(\overline{M}),\\
&&\zeta_\alpha=(e_M\otimes1)(\tilde{\eta}_\alpha)\in L^2(M)\overline{\otimes}L^2(\overline{M}),\\
&&\zeta_\alpha^\bot=\tilde{\eta}_\alpha-\zeta_\alpha\in(L^2(\tilde{M})\ominus L^2(M))\overline{\otimes}L^2(\overline{M}).
\end{eqnarray*}
Using the identity $\|(x\otimes1)\tilde{\eta}_\alpha\|_2^2 =\tau(E_M(\theta^{-1}(x^*x))) =\|x\|_2^2$ then for every $u\in F$ and a sufficiently
large $\alpha$ we obtain the following inequalities
\begin{eqnarray*}
\|[u\otimes\overline{u},\zeta_\alpha^\bot]\|_2\le\|[u\otimes\overline{u},\tilde{\eta}_\alpha]\|_2
\le\|(\theta\otimes1)([u\otimes\overline{u},\eta_\alpha])\|_2+2\|u-\theta(u)\|_2\le\frac{\epsilon}{2}.
\end{eqnarray*}
Below we proceed by contradiction to show the following inequality

\begin{eqnarray}\label{501}\text{Lim}_\alpha\|\zeta_\alpha^\bot\|_2>\delta.\end{eqnarray}

Assuming (\ref{501}) does not hold we get the following estimations:
\begin{eqnarray*}
\text{Lim}_\alpha\|\tilde{\eta}_\alpha-(e_M(\theta(v))\otimes\overline{v})\zeta_\alpha\|_2
&\le&\text{Lim}_\alpha\|\tilde{\eta}_\alpha-(e_M(\theta(v))\otimes\overline{v})\tilde{\eta}_\alpha\|_2+\text{Lim}_\alpha\|\zeta_\alpha^\bot\|_2\\
&\le&\text{Lim}_\alpha\|\tilde{\eta}_\alpha-(e_M(\theta(v))\otimes\overline{v})\tilde{\eta}_\alpha\|_2+\delta\\
&=& \text{Lim}_\alpha\|\zeta_\alpha^\bot+\zeta_\alpha-(e_M\otimes1)(\theta(v)\otimes\overline{v})\tilde{\eta}_\alpha\|_2+\delta\\
&\le&\|\zeta_\alpha^\bot\|_2+\|(e_M\otimes1)(\tilde\eta_\alpha-(\theta(v)\otimes\overline{v})\tilde{\eta}_\alpha)\|_2+\delta\\
&\le&\|\tilde{\eta}_\alpha-(\theta(v)\otimes\overline{v})\tilde{\eta}_\alpha\|_2+2\delta.
\end{eqnarray*}

Then using the above inequalities we obtain
\begin{eqnarray*}
\|E_M(\theta(v))\|_2 &\ge& \text{Lim}_\alpha\|((E_M(\theta(v)))\otimes\overline{v})\zeta_\alpha\|_2
\ge\text{Lim}_\alpha\|\tilde{\eta}_\alpha\|_2 - 2\delta\ge\sqrt{1-c^2},
\end{eqnarray*}

\noindent which obviously contradicts (\ref{502}). Thus we have shown that $\text{Lim}_\alpha\|\zeta_\alpha^\bot\|>\delta$.

 For large enough $\alpha$, the vector $\zeta=\zeta_\alpha^\bot\in \mathcal H$ satisfies
$\|\zeta\|_2\ge \delta$ and $\|[u\otimes u,\zeta]\|_2\leqslant \frac{\varepsilon}{2}$, for all $u\in F$. Also, for every $x\in M$ we have that
\begin{eqnarray*}\|(x\otimes1)\zeta\|_2=\|(x\otimes1)(e_M^\bot\otimes1)\tilde{\eta}_\alpha\|_2\\ =\|(e_M^\bot\otimes1)(x\otimes1)\tilde{\eta}_\alpha\|_2\\
\le\|(x\otimes1)\tilde{\eta}_\alpha\|_2=\|x\|_2.
\end{eqnarray*}

\noindent Using Lemma \ref{bimoduledecomp} we can view $\zeta$ as a
vector in $ (\bigoplus_i L^2(\langle M,
e_{K_i}\rangle))\overline{\otimes} L^2(M)$. Since $K_i$ is amenable
then $L^2(\langle M, e_{K_i}\rangle)$ is weakly contained in the
coarse bimodule $L^2(M)\otimes L^2(M)$. Therefore we can assume
$\zeta=(\zeta_i)_i$, with $\zeta_i\in
(L^2(M)\overline{\otimes}L^2(M))\overline{\otimes}L^2(M)$. Define
$\zeta_i'=((\id\otimes\tau)(\zeta_i\zeta_i^*))^{\frac{1}{2}}\in
L^2(M)\overline{\otimes}L^2(M)$ and $\zeta'=(\zeta_i')_i\in
\bigoplus_{i=1}^{\infty}(L^2(M)\overline{\otimes}L^2(M))$. By
proceeding exactly as in the last part of the proof of Theorem 4.9.
in \cite{OP1}, one derives that $\|x\zeta'\|_2\leqslant\|x\|_2$ for
all $x \in M$, $\|[u,\zeta']\|_2\leqslant \varepsilon$ for all $u\in
F$ and $\|\zeta'\|_2\ge \delta$. But then Corollary 2.3. in
\cite{OP1} shows that $\mathcal G''$ is amenable.

 \vskip 0.03in
So now we are left to deal with the case when $\theta_t$ does converges uniformly on $(P)_1$. In this case Theorem \ref{intertwiningrigid2}
implies that $P\prec_M A\rtimes\G$ or $P\prec_M A\rtimes H^F$ for some finite set $F\subset\G$. Since the first case already gives one of the
conclusions of our theorem, for the remaining part we assume that $P\nprec_M A\rtimes\G$ and $P\prec_M A\rtimes H^F.$

Since $P\nprec_M A\rtimes\G$ then $P\nprec_M A$. Since $P\prec_M A\rtimes H^F$, after cutting by a projection, $p$, and applying a homomorphism
we can assume $pPp\subset A\rtimes H^F$. Since $P\nprec_M A$, in fact since $P\nsubseteq A$ there is $x=\sum_{g\in H^F}x_gu_g\in P$ such that,
for some $g\in H^F$, $x_g\neq0$.

Let $y=\sum_{\gamma\in\G}y_\gamma u_\gamma\in\mathcal{N}_{pMp}(pPp)=Q$, with $y_\gamma\in A\rtimes H^\G$. Now since $yxy^*\in A\rtimes H^F$ we
must have that there is a finite set $K\subset\G$ such that $y_\gamma=0$ for $\gamma\in\G\setminus K$. Thus for all $u\in\mathcal{U}(Q)$,
$\sum_{\gamma\in K}\|E_{A\rtimes H^\G}(uu_\gamma)\|_2>1/2.$

Thus $Q\prec A\rtimes H^\G$, which is amenable since $H^\G$ is amenable. Thus we have that $Q$ is amenable as desired.

\end{proof}

When combined with results from previous section, this technical result allows us to derive a strong $W^*$-rigidity result for compact
actions of certain wreath product groups. To introduce the result let us recall first the following definition.

\begin{defn}
Let  $\G\curvearrowright X$ and $\Lambda\curvearrowright Y$ be two free, ergodic actions. We say that they are \emph{virtually conjugate} if one can find finite index subgroups, $\G_1\subset\G$ and
$\Lambda_1\subset\Lambda$, positive measure subsets $X_1\subset X$ and $Y_1\subset Y$ with $X_1$ being $\G_1$-invariant and $Y_1$ being
$\Lambda_1$-invariant such that the restrictions $\G_1\curvearrowright X_1$ and $\Lambda_1\curvearrowright Y_1$ are conjugate.
\end{defn}

\begin{thm}
Let $H,K$ be amenable groups and $\G,\La$ groups with the property (T).
Assume that $H\wr\G\curvearrowright^{\sigma}X$ and
$K\wr\La\curvearrowright^{\rho}Y$ are free, measure preserving action such that ${\sigma}_{|\G}$ is compact, ergodic and $\rho_{|\La}$ is
ergodic. If  $L^\infty(X)\rtimes (H\wr\G)\simeq L^\infty(Y)\rtimes (K\wr\La)$, then
$\G\curvearrowright^{\sigma_{{|_{\G}}}}X$ is virtually conjugate to $\La\curvearrowright^{\rho_{{|_{\La}}}}Y$.
\end{thm}

\begin{proof}
Denote by $M=L^{\infty}(X)\rtimes_{\sigma} (H\wr\G)$ and
$N=L^{\infty}(Y)\rtimes_{\rho} (K\wr\La)$. By assumption there
exists a $*$-isomorphism $\theta$ between $M$ and $N$ and since
$\sigma$ is compact then $\theta(L^{\infty}(X))\subset N$ is a
weakly compact embedding. Noticing that $\theta(L^{\infty}(X))$ is
regular in $N$, the second part of the Theorem \ref{wcint} implies
that $\theta(L^{\infty}(X))\prec_{N}L^{\infty}(Y)$. Furthermore,
since both $\theta(L^{\infty}(X))$ and $L^{\infty}(Y)$ are Cartan
subalgebras of $N$, one can find a unitary $u\in N$ such that
$u\theta(L^{\infty}(X))u^*=L^{\infty}(Y)$. In particular, we have
obtained that
$H\wr\G\curvearrowright^{\sigma}X\cong_{OE}K\wr\La\curvearrowright^{\rho}Y$
which, by Theorem \ref{merigid1}, implies that
$\G\curvearrowright^{\sigma_{{|_{\G}}}}X\cong_{OE}\La\curvearrowright^{\rho_{{|_{\La}}}}Y$.
Finally, the conclusion follows by applying Ioana's Cocycle
Superrigidity Theorem from \cite{Io08}.
\end{proof}

\begin{rem} Note that the requirements that $\G$ have property (T) and that $\sigma$ be compact on $\Gamma$ in
the previous theorem,
forces  $\Gamma$ to be residually finite.
Indeed, first note that since $\G$ has property (T), it is finitely generated. Also recall that if the action $\G\curvearrowright (X, \mu)$
is compact then the associated unitary representation on $L^2(X,
\mu)$ decomposes as a direct sum of finite dimensional
representations, which we denote $\bigoplus_{i\in I} (\pi_i,
\mathcal{H}_i)$. So if the action is faithful (which is the case, because it is free),
then given $g\in \G$ we
can chose $i\in I$ such that $\pi_i(g)$ is nontrivial. Since the image of $\G$
under $\pi_i$ is finite dimensional and $\G$ is finitely generated, by a theorem of Mal'cev (see \cite{Ma40}), the group $\pi_i(\G)$ is residually finite. Thus there is a
finite group $G_{i,g}$ and a homomorphism
$\phi_{i,g}:\pi_i(\G) \rightarrow G_{i,g}$ such that $\phi_{i,g}\circ\pi_i(g)$
is non trivial. Thus $\G$ has a finite quotient $\phi_{i,g}\circ \pi_i(\G)$ in which the image of $g$ is non-trivial,
showing that $\G$ is residually finite. Note also that if $H$ is a residually finite abelian group 
(e.g. if it is finitely generated abelian), then $H\wr \G$ follows residually finite as well
(see e.g. \cite{Gr57}). Finally, in order to see that there are many actions of wreath product groups verifying the
conditions in 6.4, note that if $H\wr\G$ is residually finite then it has  profinite (thus compact) actions.
Altogether, we can take $\G$ to be any ``classic'' Kazhdan group, like $SL(n,\mathbb Z)$, $n \geq 3$,
and $H$ to be any finitely generated abelian group, like $\mathbb Z^k$, $(\mathbb Z/m\mathbb Z)^k$, etc.
\end{rem}

\end{document}